\documentclass[12pt, dvipsnames]{amsart}

\usepackage[english]{babel}

\usepackage[T1]{fontenc}

\usepackage{tikz}
\definecolor{my1}{RGB}{30,144,255}
\definecolor{my2}{RGB}{255,20,147}
\usepackage[colorlinks=true,linkcolor=my2, citecolor=my1]{hyperref}
\usepackage{amsmath, amsfonts, mathrsfs, amsthm, amstext,amssymb}
\usepackage{enumerate}
\usepackage{array}
\usepackage{enumitem}
\usepackage{hologo}
\usepackage{graphicx}
\usepackage{multirow}
\usepackage[colorinlistoftodos]{todonotes}

\makeatletter
\DeclareFontEncoding{LS1}{}{}
\DeclareFontSubstitution{LS1}{stix}{m}{n}
\makeatother

\usepackage[hmargin=2.3cm, vmargin=3cm]{geometry}
 
\usepackage{Alegreya}      
\usepackage{AlegreyaSans}  

\title{Brownian motion on the unitary quantum group. Construction and cutoff.}
\author{Jean DELHAYE}
\thanks{}

\renewcommand{\and}{\quad \hbox{\&}\quad}
\newcommand{\et}{\quad \hbox{\&}\quad}

\newtheorem{thm}{Theorem}[section]
\newtheorem{lem}[thm]{Lemma}
\newtheorem{prop}[thm]{Proposition}

\newtheorem{thmB}{Theorem}[section]
\newtheorem{lemB}[thmB]{Lemma}
\newtheorem{propB}[thmB]{Proposition}

\theoremstyle{definition}

\newtheorem{de}[thm]{Definition}

\newtheorem{rem}[thm]{Remark}
\newtheorem{deB}[thmB]{Definition}

\newtheorem{remB}[thmB]{Remark}

\DeclareMathOperator{\Meix}{Meix}

\DeclareMathOperator{\id}{id}

\DeclareMathOperator{\Tr}{Tr}

\newcommand{\sg}{>}
\renewcommand{\sl}{<}
\newcommand{\vphi}{\varphi}

\newcommand{\C}{\mathbb{C}}

\newcommand{\E}{\mathbb{E}}
\newcommand{\p}{\mathbb{P}}
\newcommand{\F}{\mathbb{F}}
\newcommand{\G}{\mathbb{G}}
\newcommand{\T}{\mathbb{T}}
\newcommand{\N}{\mathbb{N}}
\newcommand{\R}{\mathbb{R}}
\newcommand{\Z}{\mathbb{Z}}

\newcommand{\longto}{\longrightarrow}

\renewcommand{\d}{\mathrm{d}}

\renewcommand{\O}{\mathcal{O}}

\counterwithin{equation}{section}

\allowdisplaybreaks

\usepackage{fancyhdr}
\begin{document}

\begin{abstract}
    In this study, we construct an analog of the Brownian motion on free unitary quantum groups and compute its cutoff profile.
\end{abstract}


\maketitle

{\hypersetup{hidelinks}
\tableofcontents}

\section{Introduction}

Given a sequence of compact groups $(G_N)_{N \in\N}$ each equipped with a Markov process $(X_t^{(N)})_{t \geq 0}$, we say that it exhibits \textit{cutoff} at time $(t_N)$ if the total variation distance to the Haar probability measure $m_N$ drops within a window of order $o(t_N)$, in other words if 
$$
d_N\big(
t_N(1-\epsilon)
\big)\underset{N\to\infty}{\longto} 1\and 
d_N\big(
t_N(1+\epsilon)
\big)\underset{N\to\infty}{\longto} 0,\quad \epsilon\sg 0,
$$
where $d_N(t) = \sup_{A\subset G}\vert \p[X_t\in A]-m_N(A)\vert$ denotes the total variation distance between the Haar probability measure and the process at time $t$. More precisely, if there exists a continuous function $f:\R\to[0,1]$, decreasing from $1$ to $0$ such that 
$$
d_N\big( 
t_N + cs_N
\big) \underset{N\to\infty}{\longto} f(c),\quad c\in \R,
$$
then, the function $ f $ is called a \textit{cutoff profile} of the process.

The first example of cutoff phenomenon comes from the work of P. Diaconis and M. Shahshahani in the 1980s on random transpositions, using representation theory \cite{DS81}. This discovery sparked significant interest, leading to the identification of numerous examples across various contexts. Notable cases include the dovetail shuffle \cite{Ald83, BD92}, Glauber dynamics \cite{LLP10}, random walks on random graphs \cite{LS10, BLPS18}, as well as the Brownian motion on simple compact Lie groups \cite{Mel14}. In this paper, we focus on compact quantum groups, an object initially introduced by S.L. Woronowicz in \cite{Wor98}. The representation theory of compact groups extends naturally to the quantum setting, allowing for the application of similar techniques. This has been studied extensively, starting with finite quantum groups, as explored by J.P. McCarthy in \cite{Mcc17, Mcc19}, and extended to infinite compact quantum groups in works such as \cite{Fre18, Fre19}.

While determining a cutoff profile is inherently challenging, there are several processes for which such profiles have been established. For instance, in the case of discrete processes such as random walks on finite groups, the random transposition shuffle \cite{Tey19} exhibits the following asymptotic behavior:
$$
d_N\left(\frac{N}{2} (\ln N + c) \right) \underset{N \to \infty}{\longrightarrow} d_{\mathrm{TV}}\left(\mathrm{Pois}(1 + e^{-c}), \mathrm{Pois}(1)\right), \quad c \in \mathbb{R},
$$
Other examples include the lazy random walk on the hypercube \cite{DGM90}, the simple exclusion process on the circle \cite{Lac16}, the asymmetric simple exclusion process on a segment \cite{BN22}, the Ehrenfest Urn model with multiple urns \cite{NT22}, the Gibbs Sampler \cite{NT22}, random cycles \cite{NT22} and conjugacy invariant random walks on the symmetric group \cite{OTT25}.

Very few examples of cutoffs are known on continuous objects (an example other than the Brownian motion on Lie groups is non-negatively curved diffusions \cite{PS25}), and the only objects on which cutoff profiles could be computed are quantum groups. This was done in \cite{FTW21} where the authors compute the profile for the Brownian motion on the quantum orthogonal group $O_N^+$ and the quantum permutation group $S_N^+$, as well as for quantum random transpositions.

While the Brownian motion is the most natural continuous process on simple compact Lie groups, defining a clear analog on most compact quantum groups is not straightforward. The orthogonal and permutation quantum groups provide natural processes that can be interpreted as such (coming from analogies between the central generating functionals on $O_N^+$ and $S_N^+$, as proved in \cite[Thm 10.2]{CFK14} and \cite[Thm 10.10]{FKS16}, and the classical case for compact Lie groups in \cite{Lia04}). However, extending this concept to other compact quantum groups presents significant challenges.

Our main result is the identification and analysis of the Brownian motion on the unitary quantum group $U_N^+$, which is detailed in Section \ref{zlfnkzeafnlzef} and Sections \ref{MC}-\ref{limprof} respectively, along with the computation of its right limit profile. Specifically, we establish in Theorem \ref{thisisit} that, for appropriate extensions $\widetilde{d}_N$ of the distances $d_N$ in the quantum setting, we have
$$
\widetilde{d}_N\left(N \ln(\sqrt{2} N) + cN \right) \underset{N \to \infty}{\longrightarrow} d_{\mathrm{TV}}\left(\eta_c^r, \nu_{\mathrm{SC}}\right), \quad c \geq 0,
$$
$$
\underset{N \to \infty}{\lim \sup}\, \widetilde{d}_N\left(N \ln(\sqrt{2} N) + cN \right) \geq d_{\mathrm{TV}}\left(\eta_c^r, \nu_{\mathrm{SC}}\right), \quad c \in \R,
$$
where $\nu_{\mathrm{SC}}$ denotes the semicircle distribution and $\eta_c^r$ is a mixing measure from a family of (retranslated) free Meixner distributions (see Subsection \ref{cutoff phenomenon} for definitions).

In most cases, the cutoff profile typically takes the form
$$
c \longmapsto d_{\mathrm{TV}}\left(\nu_c, \mathrm{Haar}_\infty\right), \quad c \in \mathbb{R},
$$
where $\nu_c$ represents a limit measure (or state) for the process at time $t_N + cs_N$. The main difficulty in the quantum setting arises because the process is not absolutely continuous with respect to the stationary distribution whenever $c < c_0$ (for some $c_0 \in \R\cup \{\infty\}$), depending on the scaling of the sequences). This phenomenon, noted in \cite{Bia08, FTW21}, illustrates that processes exhibit an atom in certain regions. In our case, the central algebra of the quantum unitary group $U_N^+$ is not commutative, unlike $O_N^+$ and $S_N^+$. Additionally, in the region where absolute continuity is lost, the singular part does not reduce to a single atom, but rather involves a more complex structure that we were unable to fully identify, which is why we can only establish a lower bound in this region.

\subsection*{Acknowledgment}
We express our deepest gratitude to our advisor, A. Freslon, for proposing this problem and for his invaluable guidance throughout this work. We are also sincerely thankful to A. Skalski and L. Teyssier for their insightful feedback on earlier versions of this article, which helped us address inconsistencies and clarify certain points.
\section{Preliminaries}

In this paper, we will focus on compact quantum groups, which may be unfamiliar to probabilist readers. To address this, we will dedicate a preliminary section to defining and presenting key aspects. Specifically, we will introduce free unitary and orthogonal quantum groups, along with relevant results regarding Lévy processes on these objects. A natural approach to studying the cutoff phenomenon on compact quantum groups is through their representation theory, following classical techniques used for both compact Lie groups and finite groups. Notably, the representation theory of compact quantum Lie groups exhibits structural similarities to that of these classical groups, as it arises from combinatorial constructions involving colored non-crossing partitions (see, for instance, \cite[Chap 4]{fre23}). Our exposition is adapted from \cite{FTW21} where definitions were formulated for the quantum orthogonal group $O_N^+$.

\subsection{Free unitary quantum groups}

Free unitary quantum groups, a specific class of compact quantum groups, were introduced by S. Wang in \cite{Wan95}. Quantum groups, more broadly, were originally defined by S.L. Woronowicz \cite{Wor98}. These structures involve C*-algebras, fitting their noncommutative topological nature. However, in this article, we will present an alternative definition that emphasizes the algebraic aspects, making it more accessible to non-expert readers.

\begin{de}
We define $\O(U_N^+)$ to be the universal $*$-algebra generated by $N^2$ elements $\{u_{ij}\}_{1\le i,j\le N}$ such that 
$$
\sum_{k=1}^N u_{ik} u_{jk}^* = \delta_{ij}= \sum_{k=1}^N u_{ki}^* u_{kj} \and 
\sum_{k=1}^N u_{ik}^* u_{jk} = \delta_{ij}= \sum_{k=1}^N u_{ki} u_{kj}^* .
$$
\end{de}

The universal property ensures the existence of the \textit{coproduct} $*$-homomorphism $\Delta: \O(U_N^+)\to \O(U_N^+)\otimes \O(U_N^+)$ defined by 
$$
\Delta(u_{ij}) = \sum_{k=1}^N u_{ik}\otimes u_{kj},\quad 1\le i,j\le N,
$$
encoding the ``group law" of the compact quantum group. Additionally, let us outline the existence of the \textit{counit} and the \textit{antipode}, which are crucial objects. The counit is given by the $*$-homomorphism $\varepsilon: \O(U_N^+) \to \C$, defined by $\varepsilon(u_{ij}) = \delta_{ij}$. The antipode, on the other hand, is the $*$-antihomomorphism $S: \O(U_N^+) \to \O(U_N^+)$, defined by $S(u_{ij}) = u_{ji}^*$. These maps serve as the unit and the inverse map in the noncommutative setting, respectively.

In this context, probability measures can be generalized by identifying them with their integration linear forms. These forms correspond to \textit{states}, which are unital positive linear forms on $\O(U_N^+)$. Notably, there exists a particular state that serves as the analog of the Haar measure on $U_N^+$ (see \cite{Wor98}).

\begin{thm}
There is a unique state $h$ on $U_N^+$ such that 
$$
(\id \otimes h) \circ \Delta(x) = h(x) \otimes 1 = (h \otimes \id) \circ \Delta(x),\quad x\in \O(U_N^+).
$$
It is called the \textit{Haar state} of $U_N^+$. 
\end{thm}

\begin{rem} 
Let us point out the fact that the Haar state is \textit{tracial}, i.e. we have $h(xy) = h(yx)$ for any $x,y \in \O(U_N^+)$.
\end{rem}

The use of representation theory as a powerful tool to investigate the asymptotic behavior of random walks on groups has been established since the seminal works of P. Diaconis and his coauthors (see, for example, \cite[Chap 4]{Dia88}). To describe the representations of $U_N^+$, we will make use of the \textit{free orthogonal quantum group} $O_N^+$, given by the $*$-algebra that is the quotient $\O(O_N^+)\simeq \O(U_N^+)/\{u_{ij}-u_{ij}^*\}$. We will denote by $o_{ij}$ the image of $u_{ij}$ under the quotient map $\O(U_N^+)\to \O(O_N^+)$. The coproduct factors through the quotient map yielding a compact quantum group structure.

If $z$ denotes the identity function on the unit circle $\T$, then it was proven in \cite[Prop 7]{Ban97} that the map $u_{ij}\mapsto o_{ij}z$ extends to an isomorphism of compact quantum groups between $\O(U_N^+)$ and its image in $\O(O_N^+)*\O(\T)$. T. Banica proved in \cite{Ban96} that the characters of the irreducible representations of $O_N^+$ may be labeled by the nonnegative integers $(\chi_n)_{n\in\N}$ such that
$$
\chi_0 = 1,\quad \chi_1 = \sum_{j=1}^N o_{jj} \et \chi_1\chi_n = \chi_{n+1}+\chi_{n-1},\quad n\in \N^*.
$$
This relation is reminiscent of that of \textit{Chebyshev polynomials of the second kind}. That is the sequence of polynomials $(P_n)_{n\in \N}$ recursively given by
    $$
    P_0 = 1,\quad P_1 = X \et 
    XP_{n} = P_{n+1}+P_{n-1},\quad n\in \N^*.
    $$
Note that we have the relation $\chi_n = P_n(\chi_1)$ for all $n\in \N$. 

It then follows from the description of the representation theory of free products \cite{Wan95} that the characters of $U_N^+$ can be recovered as products of characters of $O_N^+$ and powers of $z$. A precise description is given in \cite[Prop 4.3]{VV13} which we reproduce here.

\begin{thm}\label{clownclown}
The characters of the nontrivial irreducible representations of $U_N^+$ are the elements of the form 
    $$
    \chi^\epsilon_{\mathbf n} := 
    z^{[\epsilon_1]_-} \chi_{n_1} z^{\epsilon_2} \cdots  z^{\epsilon_{p}} \chi_{n_p}z^{[\epsilon_{p+1}]_+},
    $$
    where $\mathbf n = (n_1,\cdots ,n_p)\in \N^{*p}$, $\epsilon = \epsilon_1\in \{\pm1\}$, $[\epsilon]_- = \min(\epsilon,0)$, $[\epsilon]_+ = \max(\epsilon,0)$ and $\epsilon_{i+1} = (-1)^{n_i+1}\epsilon_i$. The dimensions can be recovered by applying the counit
    $$
    d_{\mathbf n} :=
    \varepsilon\big(
     z^{[\epsilon_1]_-} \chi_{n_1} z^{\epsilon_2} \cdots  z^{\epsilon_{p}} \chi_{n_p}z^{[\epsilon_{p+1}]_+}
    \big) = 
    P_{n_1}(N) \cdots  P_{n_p}(N).
    $$
\end{thm}

We denote by $\O(U_N^+)_0$ the algebra generated by the characters and call it the \textit{central algebra} of $U_N^+$. An important feature of this subalgebra is the existence of a conditional expectation (see \cite[Chap 9]{AP}) $\E: \O(U_N^+)\to \O(U_N^+)_0$ that leaves the Haar state invariant, let us recall the definition of such an object.

\begin{de}
Let $A$ be a $C^*$-algebra (that is a Banach $*$-algebra such that $\Vert a^*a\Vert = \Vert a\Vert^2$ for all $a\in A$) and $B\subset A$ a subalgebra of $A$. We call \textit{conditional expectation} from $A$ onto $B$ a linear map $\E:A\to B$ satisfying:
\begin{enumerate}[label = (\roman*)]
    \item $\E[A_+]\subset B_+$ where $A_+ = \{a^*a : a\in A\}$ and $B_+ = A_+\cap B$;
    \item $\E[b] = b$ for all $b\in B$;
    \item $\E[bab'] = b\E[a]b'$ for all $a\in A$ and $b,b'\in B$.
\end{enumerate}
\end{de}

Note that a conditional expectation is always a norm-one projection. The importance of such an object will become apparent later on.

Let us now describe the quantum analogue of a Lévy process on $U_N^+$. We call \textit{(quantum) Lévy process} on $U_N^+$ any right-continuous convolution semigroup of states, i.e. a family $(\psi_t)_{t\ge0}$ of states on $U_N^+$ such that 
\begin{itemize}
    \item $\psi_0 = \varepsilon$;
    \item $\psi_t\star \psi_s := (\psi_t\otimes \psi_s)\circ \Delta = \psi_{t+s}$;
    \item $\psi_t\to \psi_0$ weakly as $t\to0$.
\end{itemize}
 
We call \textit{generating functional} on $U_N^+$ an hermitian functional $L:\O(U_N^+)\to\C$ that vanishes on $1$ and is positive on $\ker \varepsilon$. There is a one-to-one correspondence between Lévy processes and generating functionals (see \cite[Sec 1.5]{FS16}) via the formulas 
\begin{align*}
    L &= 
    \lim_{t\to 0} \frac{\psi_t-\varepsilon}{t}, \\
    \psi_t &= 
    \exp_\star(tL) = 
    \sum_{n\ge0} 
    \frac{(tL)^{\star n}}{n!} ,\quad t\ge0.
\end{align*}

\begin{de}
We will say that a Lévy process $(\psi_t)_{t\ge0}$ is \textit{central} if all the $\psi_t$'s are central, i.e. if $\psi_t = \psi_t \circ \E$ for all $t\ge 0$. Equivalently, a Lévy process is central if its associated generating functional $L$ is \textit{central}, that is, $L = L\circ \E$.
\end{de}

The notion of centrality will become relevant in Section \ref{zlfnkzeafnlzef} as we introduce the quantum unitary Brownian motion as a central Lévy process. Observe that for any generating functional $ L $ on $ \O(U_N^+) $, the composition $ L \circ \mathbb{E} $ also defines a generating functional on $ \O(U_N^+) $ which we call the \textit{centralized} generating functional. Thus, we may consider centralized generating functionals and when doing so have in mind that all their information is contained within the restriction of $L$ to the central algebra $\O(U_N^+)_0$.

\subsection{The cutoff phenomenon}\label{cutoff phenomenon}

In this subsection, we introduce the concept of the \textit{cutoff phenomenon} within the framework of quantum groups. The cutoff phenomenon, a sharp transition in the convergence to equilibrium, is a significant topic in probability theory and has been widely studied in classical settings. Here, we aim to extend this understanding to the realm of quantum groups, particularly focusing on the unitary quantum group. We will provide necessary definitions, discuss the associated Lévy processes, and outline the framework used to investigate the \textit{limit profiles}.

Let us introduce a distance between states on $U_N^+$ that serves as a quantum analogue of the total variation distance. To motivate this definition, let us briefly recall the classical setting. Given a compact group $G$, its associated algebra of regular functions, denoted $\mathcal{O}(G)$, consists of polynomial functions in the matrix coefficients of finite-dimensional representations of $G$. The total variation distance between two probability measures $\mu$ and $\nu$ on $G$ is given by
$$
    d_{\mathrm{TV}}(\mu, \nu) = \sup_{A \subset G} \left\vert \mu(A) - \nu(A) \right\vert = 
    \sup_{A \subset G}
    \left\vert 
    \int_G \mathbf 1_A \d\mu 
    - 
    \int_G \mathbf 1_A\d\nu 
    \right \vert,
$$
where the supremum is taken over all Borel subsets of $G$. The issue is that, apart from the trivial ones, indicator functions $\mathbf{1}_A$ do not belong to $\mathcal{O}(G)$, but rather to a larger function space -- they exactly correspond to orthogonal projections in $\mathrm L^\infty(G)$. We therefore seek an appropriate extension of $\mathcal{O}(U_N^+)$ within a noncommutative $\mathrm L^\infty$-space. Following the classical approach, we equip $\mathcal{O}(U_N^+)$ with an inner product induced by the Haar state, given by
$$
    \langle x, y \rangle := h(xy^*),
$$
which defines the Hilbert space $\mathrm L^2(U_N^+)$. The left regular representation of $\mathcal{O}(U_N^+)$ on $\mathrm L^2(U_N^+)$, given by left multiplication (see \cite[Cor 1.7.5]{NT13} and the comments thereafter), gives raise to a von Neumann algebra, denoted by $\mathrm L^\infty(U_N^+)$, as the weak closure of its image. This provides a natural setting in which to define a quantum analogue of the total variation distance. For two states $\varphi$ and $\psi$ on $\mathrm L^\infty(U_N^+)$, we define their \textit{total variation distance} as
$$
    d_{\mathrm{TV}}(\varphi, \psi) = \sup_{p} \left\vert \varphi(p) - \psi(p) \right\vert,
$$
where the supremum is taken over all orthogonal projections in $\mathrm L^\infty(U_N^+)$. The issue in the quantum setting is that a state on $ \mathcal{O}(U_N^+) $ may not necessarily extend to a normal state on $ \mathrm{L}^\infty(U_N^+) $ (namely when absolute continuity w.r.t. the Haar state is lost). A resolution, given in \cite{FTW21}, is to consider the universal enveloping $ C^* $-algebra $ C(U_N^+) $ (see \cite[Sec II.8.3]{Bla06}), where any state on $ \mathcal{O}(U_N^+) $ uniquely extends to a state on $ C(U_N^+) $, yielding an element of the Fourier-Stieltjes algebra $ C(U_N^+)^* $, a noncommutative analogue of the measure algebra. Denoting by $ \Vert \cdot \Vert_{\mathrm{FS}} $ the Fourier-Stieltjes norm, one can show using arguments from \cite[Lem 2.6]{Fre19} that for any two states $ \varphi, \psi $ on $ U_N^+ $,  
$$
\frac{1}{2} \Vert \varphi - \psi \Vert_{\mathrm{FS}} = \sup_p \vert \varphi(p) - \psi(p) \vert,
$$
where the supremum is taken over all orthogonal projections in $ C(U_N^+)^{**} $, the topological double dual of $ C(U_N^+) $, serving as the analogue of the measure space. If $ \varphi$ and $\psi $ extend to states on $ \mathrm{L}^\infty(U_N^+) $, this supremum coincides with that over projections in $ \mathrm{L}^\infty(U_N^+) $ (see \cite[Prop 3.14]{BR17} and \cite[Lem 2.6]{Fre19}), so that we may extend the definition: 
$$
d_{\mathrm{TV}} (\varphi, \psi) := \frac{1}{2} \Vert \varphi - \psi \Vert_{\mathrm{FS}}.
$$  

\begin{de}\label{def-cutoff}
    Let $(\G_N,(\vphi^{(N)}_t)_{t\ge0})_{N\in\N}$ be a sequence of compact quantum groups each equipped with a Lévy process. We say that they exhibit a \textit{cutoff phenomenon} at time $(t_N)_{N\in \N}$ if for any $\epsilon>0$ we have 
    $$
    d_{\mathrm{TV}}( \vphi_{t_N(1-\epsilon)}^{(N)},h_N) \underset{N\to\infty}{\longto} 1
    \et 
    d_{\mathrm{TV}}( \vphi_{t_N(1+\epsilon)}^{(N)},h_N) \underset{N\to\infty}{\longto} 0.
    $$
    where $h_N$ denotes the Haar state of $\G_N$. More precisely, given a continuous function $f$ decreasing from $1$ to $0$ such that 
    $$
     d_{\mathrm{TV}}(\vphi_{t_N + cs_N},h) \underset{N\to\infty}{\longto} f(c),\quad c\in \R,
    $$
    for some sequence $(s_N)_{N\in \N}$ with $s_N = o(t_N)$, we say that $f$ is the \textit{limit profile} of the process.
\end{de}

\begin{rem}\label{uniqueness}
Note that a limit profile is unique only up to affine transformation. Specifically, suppose a given limit profile $ f $ is attained along the sequences $ (t_N, s_N)_N $. Then, for any $ a > 0 $ and $ b \in \mathbb{R} $, the limit profile $ c \mapsto f(ac + b) $ is achieved along the sequences $ (t_N + b s_N, a s_N)_N $. 
Now, consider a single process that exhibits two distinct limit profiles, $ f $ and $ f' $, along the sequences $ (t_N, s_N)_N $ and $ (t'_N, s'_N)_N $, respectively. Up to an affine transformation, we may assume that these limit profiles coincide at two points, say $ 0 $ and $ 1 $, and that $ f $ is strictly decreasing in a neighborhood of both $ 0 $ and $ 1 $ (as the function decreases from 1 to 0). Then, if $ d_N(t) $ denotes the distance between the process and the Haar state at time $ t $, we have:
$$
\lim_{N \to \infty} d_N(t_N) = f(0) = \lim_{N \to \infty} d_N(t'_N),
\and
\lim_{N \to \infty} d_N(t_N + s_N) = f(1) = \lim_{N \to \infty} d_N(t'_N + s'_N).
$$
Given that the $ d_N $'s are decreasing (see \cite[Lem 2.6]{FTW21}), it follows that $ t'_N = t_N + o(s_N) $ and $ s'_N = s_N + o(s_N) $. To see this, suppose, for instance, that the first equality does not hold. This would imply that, possibly after passing to a subsequence, we have 
$$
t'_N \ge t_N + c s_N
$$
for some $ c > 0 $ and for $ N $ sufficiently large (the inequality could be reversed and $c\sl 0$, but the argument remains the same). This inequality would further imply that 
$$
d_N(t_N) \le d_N(t_N + c s_N) \underset{N \to \infty}{\longrightarrow} f(c) < f(0),
$$
which is a contradiction. A similar argument shows that the second equality must hold as well. Hence, we conclude that $ f = f' $.
\end{rem}

In this paper, we explore the limit profile of the Brownian motion on $U_N^+$. We denote by $\nu_{\mathrm{SC}}$ the \textit{semicircle} distribution, which is defined as follows:
$$
\d \nu_{\mathrm{SC}} = \frac{\sqrt{4-x^2}}{2\pi} \mathbf{1}_{\vert x\vert \sl 2}\d x,
$$
and its $\mathrm L^2$-space has the (rescaled) Chebyshev polynomials as an orthonormal basis, as introduced just before Theorem \ref{clownclown}. For any real number $c$ and $r\in [0,\infty]$, let $\eta_c$ represent the distribution defined by
$$
\eta_c^r(A) = 
\E\left[
\left( 
1-e^{2c}R_r^{-2}
\right)_+ \mathbf 1_A(e^{-c}R_r+e^{c}R_r^{-1})
\right] + \int_A 
\E\left[
\frac{1}{e^{2c}R_r^{2}-xe^{-c}R_r + 1}
\right]\d\nu_{\mathrm{SC}}(x),
$$
where $R_r = \cos T_r$ with $T_r$ being a random variable of Gaussian distribution $\mathcal N(0,2r)$ with the conventions that $T_0 = 0$ and $T_\infty \sim \mathrm{Unif}([0,2\pi])$. Note that $\eta_c^0$ is a shifted free Meixner law:
$$
\eta_c^0 = (1-e^{2c})_+ \delta_{e^c+e^{-c}} + \frac{\d\nu_{\mathrm{SC}}(x)}{e^{2c}-xe^{-c}+1} = \Meix^+(-e^{-c},0)*\delta_{e^{-c}}.
$$
For further details on free Meixner laws, we refer the reader to \cite[Sec 2.2]{BB06}. In that regard, $\eta_c^r$ appears as a mixing measure of the family
$$
\frac{\d\nu_{\mathrm{SC}}(x)}{q^2-xq+1} + (1-q^2)_+\delta_{q+q^{-1}}
=
\Meix^+ (-q,0)*\delta_q
,\quad q\in \R^*.
$$
It is important to note that the semicircle distribution can be considered a free Meixner law with parameters $(0,0)$, emphasizing that $\eta_c^r$ is effectively a deformation of $\nu_{\mathrm{SC}}$. Moreover, it has been established in \cite[Lem 3.12]{FTW21} that $\eta_c^0$ is the unique measure satisfying
$$
\eta_c^0(P_n) = e^{-cn}, \quad n \in \mathbb{N}.
$$

\begin{rem}
    Note that one easily checks that the orthogonal family of monic polynomials for $\mathrm L^2(\eta_c^0)$ is given by
    $$
    \widehat P_n = P_n - e^{-c} P_{n-1}, \quad n \in \mathbb{N},
    $$
    with the convention that $P_{-1} = 0$. By inspecting this family and using \cite[Thm 2.1]{SY01}, one obtains the result from \cite[Lem 3.12]{FTW21} without using free cumulants.
\end{rem}

The limit profile of the Brownian motion on $U_N^+$ is determined by the function mapping each $c \in \mathbb{R}$ to the total variation distance between $\eta_c^r$ and $\nu_{\mathrm{SC}}$\footnote{Brownian motions on $U_N^+$ will emerge as a two parameter family, the constant $r$ essentially translates the quotient of those parameters.}.

\section{The Quantum Unitary Brownian Motion} \label{zlfnkzeafnlzef}

\subsection{Defining the quantum unitary Brownian motion}\label{6735433}

It is not entirely clear which generating functional on $U_N^+$ could serve as the analogue of the Brownian motion. Fortunately, a well-defined concept of Brownian motion on $O_N^+$ has been established. F. Cipriani, U. Franz, and A. Kula provided a classification of central generating functionals on $O_N^+$ in \cite[Thm 10.2]{CFK14} using an analogue of the Lévy-Khinchine formula. Specifically, they determined that all such functionals have the following form:
\begin{equation}\label{centralON+}
    \chi_n\longmapsto -bP'_n(N) - \int_{-N}^N
\frac{P_n(N)-P_n(x)}{N-x} \d\nu(x),
\end{equation}
for some $b\ge0$ and positive Borel measure $\nu$ on $[-N,N)$. \\

Comparing this formula with the one established by M. Liao for classical compact Lie groups in \cite{Lia04}, we see that the process corresponding to $\nu = 0$ plays a role analogous to the one associated to the Laplace-Beltrami operator. As a consequence, this functional is called the \textit{Brownian motion} on $O_N^+$. 

We lack a similar decomposition for central generating functionals on $U_N^+$, making it challenging to define the Brownian motion on this quantum group. Fortunately, the Brownian motion on $O_N^+$ has other properties that we can leverage to define a Brownian motion on $U_N^+$: 

\begin{enumerate}[label = (\roman*)]
    \item Centralized Gaussian generating functionals on $O_N^+$ are exactly the Brownian motions on $O_N^+$ (see Appendix \hyperref[Appendix A]{A} for a proof and a clear definition as to what is a Gaussian generating functional in the quantum setting). 
    \item Secondly, as proved in \cite[Prop 3.9]{BGJ20}, the Laplace-Beltrami operator on $ O_N $ acts as a Brownian motion on $O_N^+$ when composed with the quotient map $ \O(O_N^+) \to \O(O_N)$ and centralized.
\end{enumerate}

We will start from that last observation. Let us set some notations, we write $ U = (u_{ij})_{1\le i,j\le N} $ and $ \overline{U} = (u_{ij}^*)_{1\le i,j\le N} $. For a word $ w = w_1\cdots w_n $ over the set $\{\lozenge,\blacklozenge\}$, we define:
$$
U_w := U_{w_1}\otimes\cdots \otimes U_{w_n} \quad \text{and} \quad 
\chi_w := \chi_{U_w},
$$
where $ U_\lozenge = U $ and $ U_\blacklozenge = \overline{U} $. We also write:
\begin{itemize}
    \item $\ell(w) = n$ the length of the word;
    \item $p(w)/q(w)$ the number of occurrences of $\lozenge/\blacklozenge$ in the word.
\end{itemize}

The Lie algebra $\mathfrak{u}_N = \mathfrak {su}_N\oplus \mathfrak u_1$ of the unitary group $U_N$ consists of $N\times N$ skew-Hermitian matrices. Consider an orthonormal basis $(X_i)_{1\le i\sl N^2}$ of $\mathfrak {su}_N = \mathfrak u_N\cap \ker \Tr$ for the inner product 
$$
\langle X,Y\rangle = -\Tr(XY),
$$ 
which we complete with
$$
X_{N^2} := \frac{\mathrm{i} I_N}{\sqrt{N}} \in \mathrm{i} I_N \mathbb{R} \simeq \mathfrak{u}_1
$$
to obtain a full basis of the Lie algebra. We will allow for the abuse of notation and for any $X\in \mathfrak u_N$, denote by $X$ both the matrix and the differential operator acting on regular (even smooth) functions $f\in \O(U_N)$, defined by 
$$
Xf(x) := \left.\frac{\mathrm{d}}{\mathrm{d}t} f\left(x \exp(tX)\right)\right\vert_{t = 0}.
$$
The \emph{Laplace-Beltrami} operator is  defined as:
$$
L(f) := \sum_{i=1}^{N^2} X_i^2f(e), \quad f \in \O(U_N^+).
$$
It is a generating functional and we will always consider it through the quotient map $\O(U_N^+)\to \O(U_N) $ and conditional expectation $\E:\O(U_N^+)_0\to \O(U_N^+)$.

\begin{rem}
    Given a semisimple compact Lie group, one usually defines the Laplace–Beltrami operator via an orthonormal basis w.r.t. its negative Killing form. The group $U_N$ is not semisimple, and its negative Killing form, defined by
    $$
    B(X, Y) = -2N\, \Tr(XY) + 2\, \Tr(X)\Tr(Y), \quad X, Y \in \mathfrak{u}_N,
    $$
    is degenerate, as $\mathfrak{u}_1 \not\subseteq \ker \Tr$.
\end{rem}

\begin{prop}\label{Brownian motion on U_N}
    If $L$ denotes the Laplace-Beltrami operator on $U_N$, $\pi$ the quotient map $\O(U_N^+)\to \O(U_N)$ and $\E$ the conditional expectation $\O(U_N^+)\to \O(U_N^+)_0$, then the generating functional $\widehat L := L\circ \pi \circ \E$ is defined by 
    $$
    \widehat L(\chi_w) =
    -\Big( 
    \ell(w)\alpha + \big(p(w)-q(w)\big)^2\beta
    \Big) N^{\ell(w)-1}
    $$
    with $\alpha = N^2-1$ and 
\end{prop}

\begin{proof}
Let us denote by $ \Tr_\lozenge $ and $ \Tr_{\blacklozenge} $ the linear forms $ g \mapsto \Tr(g) $ and $ g \mapsto \overline{\Tr(g)} $, respectively. Let us first compute:
\begin{align*}
L_\lozenge &:= L(\Tr_\lozenge) = \sum_{i=1}^{N^2} \Tr(X_i^2) = -N^2, \\
L_{\lozenge\lozenge} &:= L(\Tr_\lozenge \cdot \Tr_\lozenge)
= \sum_{i=1}^{N^2} 2\left(
\Tr(e) \Tr(X_i^2) + \Tr(X_i)^2
\right) = -2N^3 - 2N, \\
L_{\lozenge\blacklozenge} &:= L(\Tr_\lozenge \cdot \Tr_{\blacklozenge}) =
\sum_{i=1}^{N^2} \left(
\Tr(e) \overline{\Tr(X_i^2)} + \overline{\Tr(e)} \Tr(X_i^2) + \left\vert \Tr(X_i) \right\vert^2
\right) \\ &= -2N^3 + 2N.
\end{align*}
This shows that $L$ is of the desired form for 
$$
\alpha = N^2 - 1 \quad \text{and} \quad \beta = 1,
$$
at least for elements $\chi_w$, where $w$ is a word of length $\le 2$. Let $ w = w_1 \cdots w_n $ be a word of length $n \ge 3$, with $p$ occurrences of $\lozenge$ and $q = n - p$ occurrences of $\blacklozenge$. Let us denote by $ W $ the product map of the $\Tr_{w_j}$'s, by $ W_k $ the same product omitting $\Tr_{w_k}$, and by $ W_{k\ell} $ the product omitting both $ \Tr_{w_k} $ and $ \Tr_{w_\ell} $. We may compute:
\begin{align*}
L(\chi_w) &= \sum_i X_i^2 W\big\vert_e = \sum_i X_i \left.\left(
\sum_k X_i(\Tr_{w_k}) W_k
\right)\right\vert_e \\
&= \sum_i \left.\left(
\sum_k X_i^2(\Tr_{w_k}) W_k +
\sum_{k \ne \ell} X_i(\Tr_{w_\ell}) X_i(\Tr_{w_k}) W_{k\ell}
\right)\right\vert_e \\
&= \sum_k W_k(e) L_{w_k} + \sum_i \sum_{k \ne \ell} W_{k\ell}(e)
\Tr_{w_k}(X_i) \Tr_{w_\ell}(X_i) \\
&= -n N^{n+1} - N^{n-2} \Bigg(
pq \left(L_{\lozenge\blacklozenge} - N(L_\lozenge + L_{\blacklozenge})\right)
+ \frac{p(p-1) + q(q-1)}{2} \left(L_{\lozenge\lozenge} - 2N L_\lozenge\right)
\Bigg) \\
&= -n N^{n+1} - N^{n-1} \left(
p(p - 1) + q(q - 1) - 2pq
\right) \\
&= -\left(
n(N^2 - 1) + (p - q)^2
\right) N^{n-1}.
\end{align*}
This finally proves the result. 
\end{proof}

We gather the results of Proposition \ref{Brownian motion on U_N}, Proposition \ref{Gaussian decompo} together with the subsequent comments, into a single theorem to highlight our definition of Brownian motion on $U_N^+$.

\begin{thm}\label{Brownian motion on U_N^+}
        Every pair $(\alpha,\beta)\in \R^2$ defines a central generating functional on $U_N^+$ through the formula 
        \begin{equation} \label{word}
            L_{\alpha\beta}:
        \chi_w \longmapsto - \Big( 
        \ell(w)\alpha + \big( 
        p(w)-q(w)
        \big)^2\beta
        \Big) N^{\ell(w)-1}.
        \end{equation}
        Moreover: 
        \begin{enumerate}[label = \emph{(\roman*)}]
            \item They exactly describe (real) centralized Gaussian generating functionals. 
            \item The centralized pulled-back Brownian motion from the classical group $U_N$ corresponds to such a generating functional. 
        \end{enumerate}
        We call such generating functionals (along with the associated Lévy processes), \emph{Brownian motions}. 
    \end{thm}

\subsection{Computing the values of the Brownian motion} 

Let $L = L_{\alpha\beta}$ be a Brownian motion on $U_N^+$ as defined in Theorem \ref{Brownian motion on U_N^+}. The goal of this subsection is to compute the values of $L $ on all characters of irreducible representations. By definition, $L$ coincides with a Gaussian generating functional on the central algebra. Therefore, we may apply Lemma \ref{gaussian} to $L $ on elements of the central algebra. Our first goal is to understand its values on elementary elements and products of two elementary elements. Specifically, we aim to determine the quantities of the form
$$
\Phi_m^{\epsilon_1} := L (z^{[\epsilon_1]_-}\chi^m_1 z^{[\epsilon_2]_+}) \et \Phi_{mn}^{\epsilon_1\eta_1} := L (z^{[\epsilon_1]_-}\chi^m_1 z^{[\epsilon_2]_+}\cdot 
    z^{[\eta_1]_-}\chi^n_1 z^{[\eta_2]_+}
    )
$$
where $m,n\in \N$, $\epsilon_1,\eta_1\in \{\pm1\}$, $\epsilon_2 = (-1)^{m+1}\epsilon_1$ and $\eta_2 = (-1)^{n+1}\eta_1$.

\begin{lem}
Let $\epsilon_1,\eta_1\in \{\pm1\}$ and $m,n \in \N$, we have the formulas 
\begin{equation}\label{1458}
    \Phi_m^{\epsilon_1} = -
    \big( 
    m\alpha + \mathfrak p_m\beta
    \big)N^{m-1},
\end{equation}
and 
\begin{equation}\label{1459}
    \Phi_{mn}^{\epsilon_1\eta_1} = 
    -\Big(({m+n})\alpha
    +\big( 
    (\mathfrak p_m-\mathfrak p_n)^2+\delta_{\epsilon_1\eta_1}\mathfrak p_{mn}^2
    \big)\beta\Big)
    N^{m+n-1}
\end{equation}
where,
$$
\mathfrak p_m = \left\{
\begin{array}{ll}
    0 & \hbox{if }m\hbox{ is even}  \\
    1 & \hbox{else}
\end{array}
\right. \emph{\and} \mathfrak p_{mn} = \left\{
\begin{array}{ll}
    0 & \hbox{if }m\hbox{ or }n\hbox{ is even}  \\
    2 & \hbox{else}
\end{array}
\right. 
$$
\end{lem}

\begin{proof}
For any word $w$ in the letters $\{\lozenge, \blacklozenge\}$, we set $d(w) := \vert p(w)-q(w)\vert$. 

Let $w$ be the word such that $z^{[\epsilon_1]_-}\chi^m_1 z^{[\epsilon_2]_+} = \chi_w$. Note that $w$ is a word of alternating letters of length $m$. Particularly, $\ell(w) = m$ and $d(w) = 
\mathfrak p_m = d(w)^2$. Now, it is clear that Formula \eqref{1458} follows from Equation \eqref{word}.

For the second formula, we will split the study according to the parity of $m$ and $n$. Let $w_m$ and $w_n$ be the words such that $z^{[\epsilon_1]_-}\chi^m_1 z^{[\epsilon_2]_+} = \chi_{w_m}$ and $z^{[\eta_1]_-}\chi^n_1 z^{[\eta_2]_+} = \chi_{w_n}$ and set $w := w_mw_n$ so that $\chi_w = 
z^{[\epsilon_1]_-}\chi^m_1 z^{[\epsilon_2]_+}\cdot 
z^{[\eta_1]_-}\chi^n_1 z^{[\eta_2]_+}
$.
\begin{itemize}
    \item If $m$ and $n$ are both even, then clearly $d(w_m) = d(w_n) = 0$ and so $d(w) = 0 = \mathfrak p_m-\mathfrak p_n$, thus Formula \eqref{1459} holds.
    \item If $m$ is even and $n$ is odd, then $d(w_m) = 0$ and $d(w_n) = 1$ and so $d(w) = 1 = \mathfrak p_n-\mathfrak p_m$ which is what we wanted. Same holds if $m$ is odd and $n$ is even.
    \item Now, if both $m$ and $n$ are odd, then $d(w_m) = d(w_n) = 1$ and $d(w)\in \{0,2\}$. As previously stated, both $w_m$ and $w_n$ are words of alternating letters, so the most present letter in $w_m$ and $w_n$ are their first letter. This gives us the following equivalence relation
    \begin{align*}
        d(w) = 2 &\iff w_m \hbox{ and }w_n\hbox{ have the same first letter} \\
        &\iff \epsilon_1 = \eta_1.
    \end{align*}
    This shows the last case.
\end{itemize}
\end{proof}

We now compute the Brownian motion on some irreducible characters.

\begin{prop}\label{1712}
Let $m,n\in \N$ and $\epsilon_1,\eta_1\in \{\pm1\}$, we have 
\begin{equation}
   \Psi_m^{\epsilon_1} := 
   L (z^{[\epsilon_1]_-}\chi_m z^{[\epsilon_2]_+})
   =-
   \alpha P_m'(N)-
   \beta
   \frac{\mathfrak p_m}{N}P_m(N),
\end{equation}
and
\begin{equation} 
\begin{split}
\Psi_{mn}^{\epsilon_1\eta_1} :&= L (
    z^{[\epsilon_1]_-}\chi_m z^{[\epsilon_2]_+}
    \cdot 
    z^{[\eta_1]_-}\chi_n z^{[\eta_2]_+}
    ) \\
 &= -
    \alpha
    (P_mP_n)'(N)
    -\beta
    \frac{(\mathfrak p_m-\mathfrak p_n)^2+\delta_{\epsilon_1\eta_1}\mathfrak p_{mn}^2}{N}
    (P_mP_n)(N)
\end{split}
\end{equation}
where $\epsilon_2 = (-1)^{m+1}\epsilon_1$ and $\eta_2 = (-1)^{n+1}\eta_1$.
\end{prop}

\begin{proof}
Let us write $P_m = a_0^m + \cdots + a_m^mX^m$.\\
Using Equation \eqref{1458} we have 
\begin{align*}
    \Psi_m^{\epsilon_1} &= L (z^{[\epsilon_1]_-}P_m(\chi_1)z^{[\epsilon_2]_+})
    \\ &= \sum_j a_j^m 
    L (z^{[\epsilon_1]_-}\chi_1^j z^{[\epsilon_2]_+}) \\
    &= -\sum_j a_j^m\Big( j\alpha + \mathfrak p_j \beta\Big)N^{j-1} \\ &=
    -
    \alpha P_m'(N) -
    \sum_j a_j^m \mathfrak p_m N^{j-1}\beta \\
    &= -
    \alpha P_m'(N) -\beta \frac{\mathfrak p_m}{N}P_m(N).
\end{align*}
Note that we have used the fact that $a_j^m = 0$ whenever $m$ and $j$ are of different parities. Now, using Equation \eqref{1459} we have 
\begin{align*}
    \Psi_{mn}^{\epsilon_1\eta_1} &= L (
    z^{[\epsilon_1]_-}P_m(\chi_1) z^{[\epsilon_2]_+}
    \cdot 
    z^{[\eta_1]_-}P_n(\chi_1) z^{[\eta_2]_+}
    ) \\ &=
    \sum_{i,j} a_i^m
    a_j^n
    L (
    z^{[\epsilon_1]_-}\chi_1^i z^{[\epsilon_2]_+}
    \cdot 
    z^{[\eta_1]_-}\chi_1^j z^{[\eta_2]_+}
    ) \\ &= -\sum_{i,j} a_i^m
    a_j^n 
    \left(
    (i+j)\alpha+ \Big({(\mathfrak p_i-\mathfrak p_j)^2+\delta_{\epsilon_1\eta_1}\mathfrak p^2_{ij})}
    \Big)\beta
    \right)N^{i+j-1} \\  &=-    
    \alpha
    (P_mP_n)'(N)
    -
    \sum_{i,j} a_i^m
    a_j^n 
    \Big({(\mathfrak p_i-\mathfrak p_j)^2 + \delta_{\epsilon_1\eta_1}\mathfrak p^2_{ij}}
    \Big)
    \beta
    N^{i+j-1}
    \\ &= -
    \alpha (P_mP_n)'(N)
    -\beta\frac{(\mathfrak p_m- \mathfrak p_n)^2+\delta_{\epsilon_1\eta_1}\mathfrak p_{mn}^2}{N} (P_mP_n)(N).
\end{align*}
\end{proof}

Before stating this subsection's final theorem, we need one last ingredient. We first need to define the \emph{parity entanglement} of a tuple $\mathbf n = (n_1,\cdots,n_p)\in \N^{*p}$, to do so let us fix some additional notations.
\begin{itemize}
    \item $\ell(\mathbf n) := p$ the length of $\mathbf n$;
    \item $\mathfrak p_{\mathbf n} := \sum_{j=1}^p \mathfrak p_{n_j}$ the amount of odd numbers among the $n_j$'s;
    \item $k_1 := \min\{j\ge 1 : \mathfrak p_{n_j} = 1\}$ and $k_{i+1} := \min\{j\sg k_i : \mathfrak p_{n_j} = 1\}$ ($1\le i\sl \mathfrak p_{\mathbf n}$).
\end{itemize}
We thus define the parity entanglement of $\mathbf n$ to be the quantity
$$
\mathfrak e_{\mathbf n} := \sum_{1\le i, j\le \mathfrak p_{\mathbf n}} (-1)^{k_j+j-(k_i+i)}.
$$

\begin{lem}\label{86461}
Let $\mathbf n = (n_1,\cdots,n_p)\in \N^{*p}$, $\epsilon =\epsilon_1\in \{\pm1\}$ and recursively define $\epsilon_{i+1} := (-1)^{n_i+1}\epsilon_i$, then 
$$
\mathfrak p_{\mathbf n} + 
\sum_{1\le i\sl j\le p}\delta_{\epsilon_i\epsilon_j}\mathfrak p_{n_in_j}^2
- 
\mathfrak p_{n_in_j} = \mathfrak e_{\mathbf n}.
$$
\end{lem}

\begin{proof}
Let $k_1,\cdots,k_{\mathfrak p_{\mathbf n}}$ denote the indices of the odd entries of $\mathbf n$. 

If $\mathfrak p_{\mathbf n} = 0$, both terms of the equality are null, if $\mathfrak p_{\mathbf n} = 1$, both terms equal $1$. We assume that $\mathfrak p_{\mathbf n}\ge 2$. First note that for any $1\le i\sl j \le \mathfrak p_{\mathbf n}$:
\begin{align*}
    (-1)^{n_{k_i}+\cdots+n_{k_{j-1}}} = (-1)^{j-i}.
\end{align*}
This is because the $n_{k_s}$'s are odd. This further implies
\begin{align*}
    \epsilon_{k_i}\epsilon_{k_j} &=  
    (-1)^{n_{k_i}+\cdots+n_{k_{j-1}}+k_j-k_i}\epsilon_{k_i}^2 = (-1)^{n_{k_i}+\cdots+n_{k_j-1}+k_j-k_i} =  (-1)^{k_j+j-(k_i+i)}.
\end{align*}
We may now compute
\begin{align*}
    \sum_{1\le i\sl j\le p}\delta_{\epsilon_i\epsilon_j}\mathfrak p_{n_in_j}^2
    - 
    \mathfrak p_{n_in_j} &= 
    \sum_{1\le i\sl j\le \mathfrak p_{\mathbf n}}
    \delta_{\epsilon_{k_i}\epsilon_{k_j}}\mathfrak p_{n_{k_i}n_{k_j}}^2 - \mathfrak p_{n_{k_i}n_{k_j}}
    \\&= \sum_{1\le i\sl j\le \mathfrak p_{\mathbf n}}
    4\delta_{\epsilon_{k_i}\epsilon_{k_j}}  - 2 \\
    &=
    2 \sum_{1\le i\sl j\le \mathfrak p_{\mathbf n}}
    \epsilon_{k_i}\epsilon_{k_j} \\ &= 2 \sum_{1\le i\sl j\le \mathfrak p_{\mathbf n}} (-1)^{k_j+j-(k_i+i)} \\&= 
    \sum_{1\le i, j\le \mathfrak p_{\mathbf n}} (-1)^{k_j+j-(k_i+i)} - 
    \sum_{1\le i\le \mathfrak p_{\mathbf n}} (-1)^{k_i+i-(k_i+i)}
    \\ &=\mathfrak e_{\mathbf n} - \mathfrak p_{\mathbf n}.
\end{align*}
From the second to third line, we have used the fact that $2\delta_{xy} = xy+1$ when $x,y$ are elements in $\{\pm1\}$. Thus, the equality holds.
\end{proof}

\begin{thm}\label{decompo Brownian U_N^+}
Let $\mathbf n = (n_1,\cdots,n_p)\in \N^{*p}$, $\epsilon\in \{\pm1\}$, then 
$$
L (\chi^\epsilon_{\mathbf n})
    = -
    \alpha P'_{\mathbf n}(N) - \beta \frac{ \mathfrak e_{\mathbf n}}{N}P_{\mathbf n}(N),
$$
where we have written $P_{\mathbf n} = P_{n_1}\cdots P_{n_p}$.
\end{thm}

\begin{proof}
Using Lemma \ref{gaussian}, Proposition \ref{1712} and Lemma \ref{86461} we have 
\begin{align*}
    L (
    \chi_{\mathbf n}^\epsilon) &=
    L (
    z^{[\epsilon_1]_-}
    \chi_{n_1}z^{[\epsilon_2]_+}\cdots z^{[\epsilon_p]_-}\chi_{n_p}
    z^{[\epsilon_{p+1}]_+}
    )\\ &\overset{\mathrm{\ref{gaussian}}}{=} 
    \sum_{1\le i\sl j\le p}\Psi_{n_in_j}^{\epsilon_i\epsilon_j}\prod_{k\ne i,j}P_{n_k}(N) - (p-2)\sum_{j=1}^p \Psi_{n_j}^{\epsilon_j}\prod_{k\ne j}P_{n_k}(N) \\
    &\overset{\ref{1712}}{=}
    \sum_{1\le i\sl j\le p}\left(-
    \alpha\frac{(P_{n_i}P_{n_j})'(N)}{(P_{n_i}P_{n_j})(N)}-\beta\frac{(\mathfrak p_{n_i}-\mathfrak p_{n_j})^2 + \delta_{\epsilon_i\epsilon_j}\mathfrak p_{n_in_j}^2}{N}
    \right) d_{\mathbf n}
    \\&\,\,\,\,\,\,\,\,\,\,\,\,\,- (p-1)\sum_{j=1}^p \left(-
    \alpha\frac{P_{n_j}'(N)}{P_{n_j}(N)}-
   \beta
   \frac{\mathfrak p_{n_j}}{N}
    \right)d_{\mathbf n}
    \\ 
    &=-
    \alpha
    \sum_{j=1}^p \frac{P_{n_j}'(N)}{P_{n_j}(N)}d_{\mathbf n}
    - \beta\sum_{1\le i\sl j\le p} \frac{\mathfrak p_{n_i}+\mathfrak p_{n_j} - 2\mathfrak p_{n_i}\mathfrak p_{n_j} + 
    \delta_{\epsilon_i\epsilon_j}\mathfrak p^2_{n_in_j}
    }{N}d_{\mathbf n} \\ &\,\,\,\,\,\,\,\,\,\,\,\,\,+
    \beta
    (p-1) \sum_{j=1}^p \frac{\mathfrak p_{n_j}}{N}d_{\mathbf n} \\ &=-
    \alpha
    \sum_{j=1}^p \frac{P_{n_j}'(N)}{P_{n_j}(N)}d_{\mathbf n} -
    \beta\frac{\mathfrak p_{\mathbf n
    }}{N}d_{\mathbf n} - \beta\sum_{1\le i\sl j\le p} \frac{\delta_{\epsilon_i\epsilon_j}\mathfrak p_{n_in_j}^2 - \mathfrak p_{n_in_j}}{N}d_{\mathbf n} \\ &\overset{\ref{86461}}{=}
    -\alpha P'_{\mathbf n}(N) -
    \beta\frac{\mathfrak e_{\mathbf n}}{N}P_{\mathbf n}(N).
\end{align*}
\end{proof}

It follows that Brownian motions on $U_N^+$ (seen as Lévy processes) are the Lévy processes of the form 
\begin{equation}\label{BMUN^+}
\psi_t^{(\alpha,\beta)} : \chi^\epsilon_{\mathbf n} \longmapsto d_{\mathbf n} \exp
\left(
-t\left(
\alpha\lambda_{\mathbf n}
+ \beta\frac{\mathfrak e_{\mathbf n}}{N}
\right)
\right),\quad t\ge 0,
\end{equation} 
for some $\alpha,\beta\ge0$, where we have written $\lambda_{\mathbf n} = P'_{\mathbf n}(N)/P_{\mathbf n}(N)$.

\section{Restricting the Study to a Smaller Algebra}\label{Cond exp}

Let us recall that the asymptotic study of the Brownian motion on the orthogonal quantum group $O_N^+$ was done in \cite{FTW21}. However, the proof heavily relies on the commutativity of the central algebra $\O(O_N^+)_0$, this is not the case for $\O(U_N^+)_0$ (note for instance that $\chi_1^2 \ne \overline z\chi_1^2z$). Our first goal will be to find a smaller commutative subalgebra on which we may restrict our study to drastically ease the computations. We proceed by understanding what would be the smallest algebra $A\subset \O(U_N^+)$ on which all the information of the limit-profile is contained. Before explicitly introducing this algebra, let us present some facts. 

For the remainder of this section, consider the Brownian motion $(\psi_t)_{t\ge0}$ on $U_N^+$ (see Formula \ref{BMUN^+}) of parameters $\alpha = 1$ and $\beta = 0 $. 

\begin{lem}\label{moment conv}
    For any tuple $\mathbf n =(n_1,\cdots ,n_p)$, we have, setting $t_N = N\ln N+cN$ for some $c\in \R$:
$$
\psi_{t_N}(\chi_{\mathbf n}^\epsilon) \underset{N\to\infty}{\longto} e^{- c\vert\mathbf n\vert},
$$
where we have written $\vert\mathbf n \vert= n_1+\cdots + n_p$.
\end{lem}

\begin{proof}
This follows from the fact that $P_{\mathbf n}$ is a monic polynomial of degree $\vert\mathbf n\vert$, from which we deduce
\begin{align*}
    d_{\mathbf n} 
    e^{-t_N\lambda_{\mathbf n}} &= N^{\vert \mathbf n\vert}(1+o(1)) \exp\left(
    -(\ln N + c) \left(
    \vert \mathbf n\vert +O(N^{-1})
    \right)
    \right) \\ &=
    e^{-c\vert \mathbf n\vert + o(1)}(1+o(1)) \\ &\underset{N\to\infty}{\longto} e^{-c\vert\mathbf n\vert}.
\end{align*}   
\end{proof}

Now, we understand that the limit-process evaluated on an irreducible character $\chi_{\mathbf n}^\epsilon$ only depends on $\vert \mathbf n\vert$, let us introduce \textit{compositions}. A \textit{composition} of $m$ is a tuple $\mathbf n = (n_1,\cdots ,n_p)$ that satisfies $\vert\mathbf n\vert := n_1+\cdots +n_p= m $. 

Let us denote by $\Pi_m$ the set of compositions of an integer $m\in \N^*$. Let us point out that $\vert \Pi_m\vert = 2^{m-1}$. We also write 
$$
\Pi_m^1 = \{\mathbf n\in \Pi_m : n_1 = 1\}\et \Pi_m^{\sg1} = \Pi_m\setminus \Pi_m^1.
$$

\begin{lem}\label{bijections}
    Given a tuple of integers $\mathbf{n} = (n_1, \cdots, n_p)$, let $\mathbf{n}^{\pm}$ denote the tuple defined by $(n_1 \pm 1, n_2, \cdots, n_p)$, where we identify $(0, \mathbf{m}) \sim \mathbf{m}$. The following assertions hold. 
    \begin{enumerate}[label = \emph{(\roman*)}]
        \item The maps
        \begin{align*}
    \left\{
    \begin{matrix}
        \Pi_m & \longto & \Pi_{m+1}^1 \\
        \mathbf n &\longmapsto &(1,\mathbf n)
    \end{matrix}
    \right., \quad \left\{
    \begin{matrix}
        \Pi_m & \longto & \Pi_{m+1}^{\sg1} \\
        \mathbf n &\longmapsto &\mathbf n^+
    \end{matrix}
    \right., \\ \left\{
    \begin{matrix}
        \Pi_{m+1}^1 & \longto & \Pi_{m} \\
        \mathbf n &\longmapsto &\mathbf n^-
    \end{matrix}
    \right.\emph{\and}  \left\{
    \begin{matrix}
        \Pi_{m+1}^{\sg1} & \longto & \Pi_{m} \\
        \mathbf n &\longmapsto &\mathbf n^-
    \end{matrix}
    \right.,
\end{align*}
        are bijections. 
        \item The fusion rule for characters on $U_N^+$ can be expressed as follows: 
        $$
\chi_1^\epsilon \chi_{\mathbf{n}}^\epsilon = \chi_{(1, \mathbf{n})}^\epsilon, \quad
\chi_1^\epsilon \chi_{\mathbf{n}}^{-\epsilon} = \chi_{\mathbf{n}^+}^\epsilon + \chi_{\mathbf{n}^-}^\epsilon,
$$
    \end{enumerate}
\end{lem}

\begin{proof}
    The first item is clear, let us focus on the second. Fix a tuple $\mathbf n = (n_1,\cdots ,n_p)$ and $\epsilon = \pm1$, writing 
    $$
    \chi^\epsilon_{\mathbf n} = z^{[\epsilon_0]_-}\chi_{n_1}z^{\epsilon_1}
    \cdots
    z^{\epsilon_{p-1}}\chi_{n_p}z^{[\epsilon_p]_+},
    $$
    we compute 
    \begin{align*}
        \chi_1^\epsilon
        \chi_{\mathbf n}^\epsilon &=
        {z^{[\epsilon]_-}\chi_1z^{[\epsilon]_+}
        z^{[\epsilon_0]_-}\chi_{n_1}z^{\epsilon_1}
    \cdots
    z^{\epsilon_{p-1}}\chi_{n_p}z^{[\epsilon_p]_+}
        } \\
        &= 
        z^{[\epsilon]_-}\chi_1z^\epsilon
        \chi_{n_1}z^{\epsilon_1}
        \cdots 
        z^{\epsilon_{p-1}}\chi_{n_p}z^{[\epsilon_p]_+} \\ &=
        \chi_{(1,\mathbf n)}^\epsilon,
    \end{align*}
    and 
    \begin{align*}
        \chi_1^{-\epsilon} 
        \chi_{\mathbf n}^\epsilon &=
        z^{[-\epsilon]_-}\chi_1z^{[-\epsilon]_+}
        z^{[\epsilon_0]_-}\chi_{n_1}z^{\epsilon_1}
    \cdots
    z^{\epsilon_{p-1}}\chi_{n_p}z^{[\epsilon_p]_+}
         \\
        &= 
        { z^{[-\epsilon]_-}\chi_1\chi_{n_1}z^{\epsilon_1}\cdots z^{\epsilon_{p-1}}\chi_{n_p}z^{[\epsilon_p]_+}
        } \\ &=
        {
        z^{[-\epsilon]_-}
        (\chi_{n_1+1} + \chi_{n_1-1})
        z^{\epsilon_1}\cdots z^{\epsilon_{p-1}}\chi_{n_p}z^{[\epsilon_p]_+}
        } \\ &= \chi_{\mathbf n^+}^{-\epsilon}
        +\chi_{\mathbf n^-}^{-\epsilon}.
    \end{align*}
\end{proof}

The last step is identifying a subalgebra that only cares about the composition type. We claim that this algebra is the one generated by 
$$
x := \frac{\chi_1^++\chi_1^-}{\sqrt 2},
$$
let us denote it by $\O(U_N^+)_{00}$. 

In the following proposition, we show that we may exhibit an orthonormal basis of $\O(U_N^+)_{00}$ for the inner product induced by the Haar state that behaves well with the composition type. This will allow us to restrict the study to this subalgebra through the construction of a conditional expectation onto the subalgebra $\O(U_N^+)_{00}$ leaving both the Haar state and the limit-profile invariant. 

\begin{prop}\label{cond exp}
    Let $x_m$ denote the renormalised sum of all characters that are compositions of $m$, i.e. 
    $$
    x_0 = 1 \emph{\and} 
    x_m = \frac{1}{\sqrt 2^m}\sum_{\substack{\mathbf n\in \Pi_m,\\\epsilon=\pm1}} 
\chi^\epsilon_{\mathbf n},\quad m\in \N^*.
    $$
    Then, the following assertions hold 
    \begin{enumerate}[label = \emph{(\roman*)}]
        \item The spectral measure of $x=x_1$ is the semicircle distribution $\nu_{\mathrm{SC}}$. In other words, when considering the isomorphism $\iota : \O(U_N^+)_{00}\to \C[X],x\mapsto X$, we have $h\circ \iota = \nu_{\mathrm{SC}}$.
        \item We have 
        $$
        x_1x_m = x_{m+1} + x_{m-1},\quad m\ge 1. 
        $$
        \item There exists a conditional expectation $\F:\O(U_N^+)\to \O(U_N^+)_{00}$ leaving the Haar state invariant and such that 
        $$
        \F[\chi_{\mathbf n}^\epsilon] = \frac{x_{\vert\mathbf n\vert}}{\sqrt 2^{\vert \mathbf n\vert}}.
        $$
    \end{enumerate}
\end{prop}

\begin{proof}
(i) The first item easily follows from the second. 

(ii) To simplify computations, we set $y_{\mathbf n} := \chi_{\mathbf n}^+ + \chi^-_{\mathbf n}$ so that we may write 
$$
x_m = \frac{1}{\sqrt 2^m} \sum_{\mathbf n\in \Pi_m} y_{\mathbf m},\quad m\ge 1.
$$
We now compute for any $m\ge2$, using both items from Lemma \ref{bijections}
\begin{align*}
    xx_m &= \frac{y_1}{\sqrt 2}\cdot \frac{1}{\sqrt 2^m}\sum_{{\mathbf n\in \Pi_m}} 
y_{\mathbf n},\quad m\in \N^* \\ &=
    \frac{1}{\sqrt 2^{m+1}} \sum_{{\mathbf n\in \Pi_m}} 
    y_1y_{\mathbf n} \\ &=
    \frac{1}{\sqrt 2^{m+1}} \sum_{{\mathbf n\in \Pi_m}} \big(
    y_{(1,\mathbf n)} + y_{\mathbf n^+} + y_{\mathbf n^-}\big) \\ &=
    \frac{1}{\sqrt 2^{m+1}} \sum_{{\mathbf n\in \Pi_m}} \big(
    y_{(1,\mathbf n)} + y_{\mathbf n^+}\big)
    + \frac{1}{\sqrt 2^{m+1}} \sum_{{\mathbf n\in \Pi_m}}   y_{\mathbf n^-} \\ &=
    x_{m+1} + \frac{2}{\sqrt 2^{m+1}} \sum_{{\mathbf n\in \Pi_{m-1}}}  y_{\mathbf n} \\ &=
    x_{m+1} + x_{m-1}. 
\end{align*}
The case $m=1$ can be done separately 
\begin{align*}
    x^2 = \frac{(\chi_1^++\chi_1^-)^2}{2}
    = \frac{\chi_{(1,1)}^+ + \chi_2^+ + 1 + \chi_2^- + 1 + \chi^-_{(1,1)}}{2} = x_2 + 1.
\end{align*}

(iii) Let $\mathrm L^2(U_N^+)_0$ and $\mathrm L^2(U_N^+)_{00}$ be the $\mathrm L^2$-spaces of $\O(U_N^+)_0$ and $\O(U_N^+)_{00}$ respectively and $\F'$ be the orthogonal projection $\mathrm L^2(U_N^+)_0\to \mathrm L^2(U_N^+)_{00}$. First, we prove that $\F'$ satisfies the relation stated in the proposition, which simply follows from the fact that the $\chi_{\mathbf n}^\epsilon$'s form an orthonormal basis -- which in turn implies that for any composition $\mathbf n$ of type $m$, one has 
$$
\langle x_m , \chi_{\mathbf n}^\epsilon \rangle = 
\frac{1}{\sqrt 2^m}
\sum_{\mathbf n'\in \Pi_m} \langle \chi_{\mathbf n'}^\epsilon + \chi_{\mathbf n'}^{-\epsilon}, \chi_{\mathbf n}^\epsilon\rangle = \frac{1}{\sqrt 2^m}.
$$
Clearly, if $\mathbf n$ is a composition of type $\ne m$, then $\langle x_m,\chi_{\mathbf n}^\epsilon\rangle =0$. Since the Haar state is faithful, the algebra $\O(U_N^+)$ embeds onto the von Neumann algebra $\mathrm L^\infty(U_N^+)$ so that it follows from \cite[Thm 9.1.2]{AP} that restricting $\F = \F'\circ\E$ gives a conditional expectation.
\end{proof}   

Note that when we say that the limit-profile is invariant under the conditional expectation $\F$, we mean that setting $t_N = N\ln N + cN$, we have 
\begin{align*}
    \lim_{N\to \infty} \psi_{t_N}\circ \F(\chi_{\mathbf n}^\epsilon) &= 
    \lim_{N\to \infty} 
    \frac{1}{2^{\vert \mathbf n\vert}}
    \sum_{\substack{\mathbf n'\in \Pi_{\vert \mathbf n\vert}
    \\
    \epsilon = \pm1
    }}
    \psi_{t_N}(\chi_{\mathbf n'}^\epsilon)
    = e^{-c\vert\mathbf n\vert}
    =
    \lim_{N\to \infty} \psi_{t_N}(\chi_{\mathbf n}^\epsilon).
\end{align*}

\begin{rem}
    The first item of this proposition was already proved in \cite{Ban97}.
\end{rem}

Finally, let us end this subsection by mentioning that the closure of $\O(U_N^+)_{00}$ in $C(U_N^+)$ is isomorphic to $C([-N,N])$. Indeed, it suffices to see that the spectrum of $x/\sqrt 2$ is $[-N,N]$. The inclusion $\sigma(x/\sqrt 2)\subset [-N,N]$ follows from the fact that $\Vert x/\sqrt 2\Vert_{C(U_N^+)}\le N$. On the other hand, there exists a surjective C*-homomorphism $C(U_N^+)\to C(O_N^+),u_{ij}\mapsto o_{ij}$. In particular, $x/\sqrt 2$ is sent to $\chi_1$ which is well-known to have spectrum $[-N,N]$ (see \cite[Lem 4.2]{Bra12}), hence the other inclusion.

\section{Moment Convergence}\label{MC}

Consider now $\psi = (\psi_t^{(\alpha,\beta)})_{t\ge0}$ a Brownian motion on $U_N^+$ of parameters $(\alpha,\beta)\in \R_+^2$. In the last section, we have intentionally considered the simple case of $\alpha = 1$ and $\beta = 0$ to give an intuition as to which subalgebra could suit our study. The goal of this section is to consider the family of states $(\psi_t)_{t \ge 0}$ on the subalgebra $\mathcal{O}(U_N^+)_{00}$, in order to exhibit a sequence $(t_N, s_N)_{N \in \mathbb{N}}$ such that, when viewing these states as measures via the isomorphism $\iota : \mathcal{O}(U_N^+)_{00} \to \mathbb{C}[X],\, x \mapsto X$, we can establish a moment convergence:
$$
\psi_{t_N+cs_N} \underset{N\to\infty}{\longto} \eta_c,\quad c\in \R. 
$$
Ideally allowing us to get a cutoff profile of the form 
$$
c\longmapsto d_{\mathrm{TV}}(\eta_c,\nu_{\mathrm{SC}}),
$$
where $\nu_{\mathrm{SC}}$ denotes the semicircle distribution (recall -- see Proposition \ref{cond exp} -- that the Haar state restricted to the subalgebra corresponds to the semicircle distribution). 

We first describe the measures that will come into play. 

\begin{lem}\label{what are the measures}
    Set the following random variables. 
    \begin{itemize}
        \item Let $(S_n)_{n\in \N}$ be a $\pm1$ walk;
        \item Let $R_r = \cos T_r$ where $T_r$ is a random variable of Gaussian distribution $\mathcal N(0,2r)$ for any $r\in \R_+$ (by convention $R_0 = 1$);
        \item Let $R_\infty = \cos \Theta$ where $\Theta$ is a random variable of uniform distribution $\mathrm{Unif}([0,2\pi])$.
    \end{itemize}
    Then, given $r\in \R_+\cup \{\infty\}$ and $c\in \R$, there exists a unique probability measure $\eta_c^r$ such that 
    $$
    \int P_n \d\eta_c^r = e^{-cn} \E[e^{-r S_n^2}],\quad n\in \N,
    $$ 
    with the convention that $\E[e^{-\infty S_n^2}] = \p[S_n=0]$. Moreover, we may decompose $\eta_c^r$ as follows: 
    $$
    \eta_c^r = f_c^r \d\nu_{\mathrm{SC}} + \eta_{c,\mathrm{sing}}^r,
    $$
    with $\eta_{c,\mathrm{sing}}^r \perp \nu_{\mathrm{SC}}$ and where 
    $$
    f_c^r(x) =
    \E\left[\frac{1}{e^{-2c}R_r^2-xe^{-c}R_r+1}\right],\quad x\in (-2,2),
    $$
    and the singular part is given by\footnote{the positive part is defined by $x_+ = \max(x,0)$ for any $x\in \R$.}
    $$
    \eta_{c,\mathrm{sing}}^r(A) = \E\left[\left(
    1-e^{2c}R_r^{-2}
    \right)_+ \mathbf 1_A(e^{-c}R_r + e^cR_r^{-1})\right]
    $$
\end{lem}

\begin{proof}
    First, note that $\nu_{\mathrm{SC}}$ and $\eta_{c,\mathrm{sing}}^r$ have disjoint supports: $\nu_{\mathrm{SC}}$ is supported on $(-2,2)$, whereas $\eta_{c,\mathrm{sing}}^r$ allocates no mass to this interval. Indeed, since $\vert e^{-c}R_r + e^{c}R_r^{-1} \vert\ge 2$ (as the function $x \mapsto \vert x + x^{-1}\vert$ is always greater or equal than $2$).

    We consider the distributions $\eta_c^r$ as decomposed in the theorem and we prove that they have the desired moments (the measures being compactly supported\footnote{Note that the singular part is supported on the interval $[-e^c-e^{-c},e^c+e^{-c}]$ as $\vert e^{-c}R_r + e^{c}R_r^{-1} \vert\le e^c+e^{-c}$.}, they are moment-determined). 

    \underline{Step 1: Case $r =0$.} For any $c\in \R$, write 
    $$
    \eta_c := \eta_c^0 = f_c\d\nu_{\mathrm{SC}} + (1-e^{2c})_+\delta_{e^c+e^{-c}},
    $$
    where 
    $$
    f_c(x) = \frac{1}{e^{-2c}-xe^{-c}+1},\quad x\in (-2,2). 
    $$
    For now assume that $c\sg 0$ so that the atomic part vanishes. Applying the recursive formula satisfied by the $P_n$'s allows to establish that 
    $$
    \sum_{n\ge 0}e^{-nc}P_n(x) = f_c(x) = \frac{1}{e^{-2c}-xe^{-c}+1},\quad x\in (-2,2). 
    $$
    Recall that the $P_n$'s form an orthonormal family for the semicircle distribution ensuring that $\eta_c = f_c\d\nu_{\mathrm{SC}}$ has the desired moments: 
    $$
    \eta_c(P_n) = e^{-nc},\quad n\in \N.
    $$
    The case $c = 0$ follows from dominated convergence: 
    \begin{align*}
        1 &= \lim_{c\downarrow 0} e^{-nc} = \lim_{c\downarrow 0} \eta_c(P_n) \\ &= 
        \lim_{c\downarrow 0} 
        \int \frac{P_n(x)}{e^{-2c}-xe^{-c}+1}\d\nu_{\mathrm{SC}}(x) \\ &= \int
        \frac{P_n(x)}{2-x}\d\nu_{\mathrm{SC}}(x) \\
        &= \eta_0(P_n). 
    \end{align*}
    Assume now that $c\sl 0$. Using the fact that $f_c = e^{2c}f_{-c}$, together with the equality (which follows from induction)
    $$
    P_n(q+q^{-1}) = \frac{q^{-(n+1)}-q^{n+1}}{q^{-1}-q},\quad q\sg0,n\in\N,
    $$
    we may establish: 
    \begin{align*}
        \eta_c(P_n) &= \nu_{\mathrm{SC}}(f_cP_n) + P_n(e^c+e^{-c}) \\
        &= e^{2c}\eta_{-c}(P_n) + \frac{e^{-(n+1)c}-e^{(n+1)c}}{e^{-c}-e^c} \\ 
        &= e^{(n+2)c}+ \frac{e^{-(n+1)c}-e^{(n+1)c}}{e^{-c}-e^c} \\
        &= \frac{e^{(n+2)c}(e^{-c}-e^c)+e^{-(n+1)c}-e^{(n+1)c}}{e^{-c}-e^c} \\ &= 
        \frac{e^{-(n+1)c}-e^{-(n-1)c}}{e^{-c}-e^c}\\
        &= e^{-nc}.
    \end{align*}
    Which proves the last case. 

    \underline{Step 2: Case $r\in (0,\infty)$.} For simplicity, write $\mu_q := \eta_{\ln q^{-1}}$ for any $q\sg 0$ (so that $\eta_c = \mu_{e^{-c}}$). Before computing, let us first recall the Fourier transform of a Gaussian:
    $$
    e^{-rs^2} = \frac{1}{\sqrt{4\pi r}} \int_{\R} 
    e^{-t^2/4r}e^{\mathrm its}
    \d t,\quad s\in \R.
    $$
    This further implies that 
    \begin{align*}
        \E[e^{-rS_n^2}] &= 
    \frac{1}{\sqrt{4\pi r}} \int_{\R} 
    e^{-t^2/4r}\E[e^{\mathrm itS_n}]
    \d t \\ 
    &=\frac{1}{\sqrt{4\pi r}} \int_{\R} 
    e^{-t^2/4r}\E[e^{\mathrm itS_1}]^n
    \d t \\ &= \frac{1}{\sqrt{4\pi r}} \int_{\R} 
    e^{-t^2/4r}\left( 
    \frac{e^{-\mathrm it} + e^{\mathrm it}}{2}
    \right)^n
    \d t \\ &= \frac{1}{\sqrt{4\pi r}} \int_{\R} 
    e^{-t^2/4r}\cos(t)^n
    \d t.
    \end{align*}
    From this, we compute: 
    \begin{align*}
        \eta_c^r(P_n) &= \frac{1}{\sqrt{4\pi r}} \int_{\R} \Bigg( 
        \int_{-2}^2 
        \frac{P_n(x)\d\nu_{\mathrm{SC}}(x)}{e^{-2c}\cos(t)^2-xe^{-c}\cos(t)+1}
        \\ &\,\,\,+ \left( 
        1-e^{2c}\cos(t)^{-2}
        \right)_+ P_n(e^{-c}\cos(t) + e^{c}\cos(t)^{-1})
        \Bigg) e^{-t^2/4r}\d t \\ &= 
        \frac{1}{\sqrt{4\pi r}} \int_{\R} 
        \mu_{e^{-c}\cos(t)}(P_n) 
        e^{-t^2/4r}\d t \\ &= 
        \frac{1}{\sqrt{4\pi r}} \int_{\R} 
        e^{-nc}\cos(t)^n
        e^{-t^2/4r}\d t \\ &= 
        e^{-nc} \E[e^{-rS_n^2}].
    \end{align*}
    \underline{Step 3. Case $r=\infty$.} Similarly to the previous case, we compute the moments to get 
    \begin{align*}
        \eta_c^\infty(P_{2k+1}) &= \frac{1}{2\pi}\int_0^{2\pi} e^{-(2k+1)c}\cos(\theta)^{2k+1}\d\theta \\ &= 0 \\ &= e^{-(2k+1)c}\p[S_{2k+1} = 0],
    \end{align*}
    and 
    \begin{align*}
        \eta_c^\infty(P_{2k}) &= \frac{1}{2\pi}\int_0^{2\pi} e^{-2kc}\cos(\theta)^{2k}\d\theta
        \\ &= \frac{e^{-2kc}}{2^{2k+1}\pi}
        \int_0^{2\pi} (e^{\mathrm i\theta}+e^{-\mathrm i\theta})^n\d\theta \\ &=
        \frac{e^{-2kc}}{2^{2k+1}\pi}
        \int_0^{2\pi} 
        \sum_{j=0}^{2k}
        \binom{2k}{j}e^{\mathrm i (j-(2k-j))\theta}\d\theta \\ &=
        \frac{e^{-2kc}}{2^{2k}} \binom{2k}{k} \\ &= 
        e^{-2kc} \p[S_{2k} = 0]. 
    \end{align*}
\end{proof}

Before proving the moment-convergence, we give details on the $\beta$-term via the following lemma. 

\begin{lem}\label{584310}
    Define the sets 
    $$
    A_{2p}^j = \{\mathbf n\in \Pi_{2p} : 
    \mathfrak e_{\mathbf n} = (2j)^2
    \}
    \emph{\and}
    A_{2p+1}^j = \{\mathbf n\in \Pi_{2p+1} : 
    \mathfrak e_{\mathbf n} = (2j+1)^2
    \},\quad 0\le j\le p.
    $$
    Then the following assertions hold. 
    \begin{enumerate}[label = \emph{(\roman*)}]
        \item We have 
        $$
        \vert A_{2p}^0\vert = \frac{1}{2}\binom{2p}{p},\quad 
        \vert A_{2p}^j\vert = \binom{2p}{p+j},\quad 1\le j\le p,\emph{\and}\vert A_{2p+1}^j\vert = \binom{2p+1}{p+j+1},\quad 0\le j\le p.
        $$
        \item For any $m\in \N$, we have 
        $$
        \Pi_m = \bigsqcup_{j=0}^{\lfloor m/2\rfloor} A_m^j. 
        $$
    \end{enumerate}
\end{lem} 

\begin{proof}
    Given a tuple $\mathbf n$, define $\mathfrak f(\mathbf n) = \sum_i (-1)^{k_i+i}$ where the $k_i$'s denote the ordered indices of the odd numbers within the tuple so that (see the notations before Lemma \ref{86461})
    $$
    \mathfrak e_{\mathbf n} = \mathfrak f(\mathbf n)^2.
    $$
    If $f_m^k$ ($-m\sl k\le m$) denotes the cardinal of the set 
    $$
    \{\mathbf n\in \Pi_m : \mathfrak f(\mathbf n) = k\}.
    $$
    Then, we claim that 
    $$
    f_m^k =
    \left\{
    \begin{array}{cl}
    \binom{m-1}{(m-k)/2} & \hbox{ if }m-k\hbox{ is even}  \\ 
    0 & \hbox{ else}.
    \end{array}
    \right.
    $$
    The key is to understand how the map $\mathbf n \mapsto \mathfrak f(\mathbf n)$ is affected by the transformations $\mathbf n \mapsto (1,\mathbf n)$ and $\mathbf n \mapsto \mathbf n^+$\footnote{Recall that if $\mathbf n = (n_1,\cdots ,n_p)$ denotes a tuple, then $\mathbf n^+$ is the tuple $(n_1+1,n_2,\cdots, n_p)$.}. Fix a tuple $\mathbf n = (n_1,\cdots ,n_p)$ and let $k_1, \cdots, k_r$ denote the ordered indices of the odd numbers within the tuple.
    \begin{itemize}
        \item The ordered indices of the odd numbers in $(1,\mathbf n)$ are given by
        $$
        k_1' = 1 \and k_{i+1}' = k_i + 1,\quad 1 \le i \le r.
        $$
        This leads to
        \begin{align*}
            \mathfrak f(1,\mathbf n) = 1 + \sum_{i=1}^{r} (-1)^{k_i + i + 2} = 1 + \mathfrak f(\mathbf n).
        \end{align*}

        \item The ordered indices of the odd numbers in $\mathbf n^+$ are given by
        \begin{align*}
            k_1' = 1 \and k_{i+1}' = k_i,\quad 1 \le i \le r,\quad &\hbox{ if } k_1 \sg 1, \\
            k_i' = k_{i+1},\quad 1 \le i \sl r, \quad & \hbox{ otherwise}.
        \end{align*}
        We compute both cases:
        \begin{align*}
            \mathfrak f(\mathbf n^+) = 1 + \sum_{i=1}^r (-1)^{k_i + i + 1} = 1 - \mathfrak f(\mathbf n),
        \end{align*}
        and
        \begin{align*}
            \mathfrak f(\mathbf n^+) = \sum_{i=1}^{r-1} (-1)^{k_{i+1} + i} = (-1)^{k_1 + 1} - \mathfrak f(\mathbf n^+) = 1 - \mathfrak f(\mathbf n).
        \end{align*}
    \end{itemize}
    Now, the equalities
    $$
    \mathfrak f(1,\mathbf n) = 1 + \mathfrak f(\mathbf n)
    \and 
    \mathfrak f(\mathbf n^+) = 1 - \mathfrak f(\mathbf n),
    $$
    together with the fact that
    $$
    \Pi_{m+1} = \{(1,\mathbf n) : \mathbf n \in \Pi_m\} \sqcup \{\mathbf n^+ : \mathbf n \in \Pi_m\}
    $$
    imply the recursive formula
    $$
    f_{m+1}^k = f_m^{k-1} + f_m^{1-k},
    $$
    with the usual convention that $f_m^k = 0$ whenever the indices are out of bound. This naturally leads to the announced expression for $f_m^k$. From this, we simply compute for $p\ge 1$ and $j\in \{1,\cdots ,p\}$:
    \begin{align*}
        \vert A_{2p}^j\vert = f_{2p}^{2j} + f_{2p}^{-2j} = 
        \binom{2p-1}{p-j} + \binom{2p-1}{p+j}  
        = \binom{2p-1}{p+j-1} + \binom{2p-1}{p+j}
        = \binom{2p}{p+j}
    \end{align*}
    and (allowing now for $j$ and $p$ to be zero), 
    $$
    \vert A_{2p+1}^j\vert = f_{2p+1}^{2j+1} + f_{2p+1}^{-2j-1} = 
        \binom{2p}{p-j} + \binom{2p}{p+j+1}  
        = \binom{2p}{p+j} + \binom{2p}{p+j+1}
        = \binom{2p+1}{p+j+1},
    $$
    and the degenerate case 
    $$
    \vert A_{2p}^0\vert = f_{2p}^0 = \binom{2p-1}{p} = 
    \frac{1}{2} \left( 
    \binom{2p-1}{p} + \binom{2p-1}{p-1}
    \right) = \frac{1}{2} \binom{2p}{p}
    $$
    For the second point, simply note that for fixed $m$, the $A_m^j$’s are disjoint, included in $\Pi_m$, and their cardinalities add up to that of $\Pi_m$. 
\end{proof}

We now establish the announced moment-convergence. 

\begin{prop}\label{measure description}
    For every $N\in \N$, let $(\psi_t)_{t\ge0}$ be a Brownian motion on $U_N^+$ of parameters $(\alpha,\beta)\in \R_+^2$ so that 
    $$
    \frac{\beta\ln N}{\alpha} \underset{N\to\infty}{\longto} r\in [0,\infty].
    $$
    Then, setting $t_N = \alpha^{-1}N\ln (\sqrt 2N)+\alpha^{-1}cN$ for any $c\in \R$ gives 
    $$
    \psi_{t_N}(x_m) \underset{N\to\infty}{\longto} 
    e^{-mc}\E[e^{-rS_n^2}]
    $$
    where $(S_n)_{n\in \N}$ is a $\pm1$ walk and the convention that $\E[e^{-\infty S_n^2}] = \p[S_n=0]$.
\end{prop}

\begin{proof}
    We set the convention $e^{-\infty s} = 0$ if $s\sg0$ and $1$ if $s = 0$, we compute using Lemmas \ref{moment conv} and \ref{584310}: 
    \begin{align*}
        \psi_{t_N}(x_{2p}) &= \frac{1}{\sqrt 2^{2p}} \sum_{\substack{\mathbf n\in \Pi_{2p}
        \\
        \epsilon = \pm1
        }} \psi_{t_N}(\chi_{\mathbf{n}}^\epsilon) \\
        &= \frac{1}{2^{p-1}} \sum_{\mathbf n\in \Pi_{2p}} d_{\mathbf n}\exp \left( 
        -t_N \left( 
        \alpha\lambda_{\mathbf n} + \beta \frac{\mathfrak e_{\mathbf n}}{N}
        \right)
        \right) \\ 
        &= \frac{1}{2^{p-1}} 
        \sum_{j=0}^p
        \sum_{\mathbf n\in S^j_{2p}} d_{\mathbf n}
        \exp \left(-N(\ln N + (\ln \sqrt 2 + c))\lambda_{\mathbf n}\right)
        \exp\left( -
        \frac{\beta\ln N}{\alpha}(2j)^2
        \left(
        1+ o(1) 
        \right)
        \right) \\ 
        &\underset{N\to\infty}{\longto}
        \frac{1}{2^{p-1}} 
        \sum_{j=0}^p
        \sum_{\mathbf n\in S^j_{2p}} d_{\mathbf n}
        e^{-2p(c+\ln\sqrt{2})} e^{-r(2j)^2} \\ 
        &= 
        \frac{e^{-2pc}}{2^{2p-1}} 
        \left(
        \frac{1}{2}\binom{2p}{p}+
        \sum_{j=1}^p \binom{2p}{p+j} e^{-r(2j)^2}
        \right)
        \\ 
        &= \frac{e^{-2pc}}{2^{2p}} \sum_{-p\le j\le p} \binom{2p}{p+j} e^{-r(2j)^2} \\
        &= e^{-2pc}\E[e^{-rS_{2p}^2}].
    \end{align*}
    Similarly, 
    \begin{align*}
        \psi_{t_N}(x_{2p+1}) &\underset{N\to\infty}{\longto} \frac{e^{-(2p+1)c}}{ 2^{2p}} \sum_{j=0}^p \binom{2p+1}{p+j+1}
        e^{-r(2j+1)^2}
        \\ 
        &= \frac{e^{-(2p+1)c}}{ 2^{2p}}
        \sum_{-p-1\le j\le p} \binom{2p+1}{p+j+1}
        e^{-r(2j+1)^2} \\ &=
        e^{-(2p+1)c} \E[e^{-rS_n^2}].
    \end{align*}
\end{proof}

\section{A Special Case Without Cutoff}

We first eliminate a special case in which there is no cutoff, that is the case of Brownian motions on $U_N^+$ with parameters $(\alpha,\beta)$ where $\alpha = 0$.

\begin{thm}
    Let $\beta \sg 0$. Denote by $\psi^\beta = (\psi^\beta_t)_{t\ge0}$ the Brownian motion of parameters $(0, \beta)$ on $U_N^+$. This Lévy process exhibits no cutoff on the whole quantum group. More precisely,
    $$
    d_{\mathrm{TV}}(\psi_t^\beta, h) = 1, \quad t \geq 0.
    $$
\end{thm}

\begin{proof}
    Let $L_\beta$ denote the associated generating functional. 

    It suffices to note that $L_{\beta}(\chi_{\mathbf n}^\epsilon) = 0$ whenever $\mathbf n$ is a tuple consisting only of even integers. In particular, consider the commutative subalgebra $\C[\chi_2^+]$ (note that $\chi_2^+$ is selfadjoint) -- the restriction to this subalgebra allows for the isomorphism:
    $$
    \left\{
    \begin{matrix}
    \mathbb{C}[\chi_2^+] & \longrightarrow & \mathbb{C}[X] \\ 
    \chi_{2n}^+ & \longmapsto & P_{2n}(\sqrt{X})
    \end{matrix}
    \right..
    $$
    To see this, we simply need to show that the fusion rules translate from both viewpoints, applying the fusion rule from Lemma \ref{bijections} and the recursion satisfied by the Chebyshev polynomials of the second kind\footnote{We recall, once again, that they are defined by $P_0 =1$, $P_1 = X$ and $XP_n = P_{n+1}+P_{n-1}$ for $n\ge1$.}. Fix $n\ge 1$ and compute:
    \begin{align*}
        \chi_2^+ \chi_{2n}^+ &= (\chi_1^+\chi_1^- - 1)\chi_{2n}^+ \\ &=
        \chi_1^+(\chi_{2n+1}^-+\chi_{2n-1}^-)-\chi_{2n}^+
        \\ &= \chi_{2n+2}^+ + \chi_{2n}^+ + \chi_{2n-2}^+,
    \end{align*}
    and 
    \begin{align*}
        P_2(\sqrt X) P_{2n}(\sqrt X) &= (X-1)P_{2n}(\sqrt X) \\ &= \sqrt X \left( 
        P_{2n+1}(\sqrt X) + P_{2n-1}(\sqrt X)
        \right) - P_{2n}(\sqrt X) \\ 
        &= P_{2n+2}(\sqrt X) + P_{2n}(\sqrt X) + P_{2n-2}(\sqrt X).
    \end{align*}
    
        Viewed as a probability measure through this mapping, the Haar state corresponds to a squared semicircle law (which is a translated free Poisson law with parameters $(1,1)$, for more information on these families of distributions, see for instance \cite[Def 12.12]{NS06}), whose orthonormal basis is given by the $P_{2n}(\sqrt{X})$'s, and thus defines a continuous distribution. On the other hand, the Brownian motion satisfies for all $n\in \N$:
    $$
    \psi_t^{\beta}\left(P_{2n}(\sqrt{X})\right) = P_{2n}(N) 
    \exp\left(
    tL_{\beta}(\chi_{2n}^+)
    \right) = P_{2n}(N).
    $$
    It follows that $\psi_t^\beta$ corresponds to the Dirac mass $\delta_{N^2}$, which in turn implies the claimed total variation distance.
\end{proof}

The Lévy process remains, in total variation distance, as far from the Haar state as possible. As this might not be obvious from the result alone, let us emphasize that it stems from the fact that the process does not ``shuffle everywhere". Proper convergence of the Brownian motion requires the contribution carried by the $\alpha$-term.

This phenomenon of “lack of shuffling” can be understood more precisely in the quantum setting, where the Lévy process agrees with the counit on large subalgebras.

Let us now present a neat algebraic result that clarifies where the shuffling fails to occur.

\begin{prop}
    Let $\beta \ge0$ and $\psi^\beta$ be the Brownian motion of parameters $(0,\beta)$ on $U_N^+$ and $L_\beta$ the associated generating functional. Then, $\psi_t^\beta$ agrees with the counit for every $t\ge 0$ on the (unital) subalgebra 
    $$
    A = \langle 
    \chi_{\mathbf n}^\epsilon : L_{0\beta}(\chi_{\mathbf n}^\epsilon) = 0
    \rangle.
    $$
\end{prop}

\begin{proof}
    The only non-obvious fact is that $A$ is an algebra -- more precisely, that it is closed under the multiplication.

    Given a tuple of integers $\mathbf n = (n_1, \cdots, n_p)$, look at the odd numbers that appear in the tuple. Label their ordered indices $k_1, \cdots, k_r$, then define:
    $$
    f({k_i}) := (-1)^{k_i + i}, \quad 1 \le i \le r.
    $$
    It follows from the proof of Lemma \ref{estimequad} that $L_{\beta}(\chi_{\mathbf n}^\epsilon) = 0$ if and only if there are as many ones and minus ones among the $f({k_i})$'s. Or equivalently, if there exists an involutive, fully supported permutation $\sigma$\footnote{i.e., the permutation has no fixed point and $\sigma^2 = \id$.} of the set\footnote{we will allow the abuse of notation $\sigma(n_i)$/$f(n_i)$ instead of $\sigma(i)$/$f(i)$.}
    $$
    \mathrm{Odd}(\mathbf n) = \{i : n_i\hbox{ is odd}\},
    $$
    such that $f\circ \sigma+f=0$. Such a permutation will be referred to as an \emph{admissible} permutation of the tuple, and the corresponding tuple will likewise be called \emph{admissible} if such a permutation exists.
    
    Now let $\mathbf m = (m_1, \cdots, m_p)$ and $\mathbf n = (n_1, \cdots, n_q)$ be two admissible tuples of integers, together with associated admissible permutations $\sigma_{\mathbf m}$ and $\sigma_{\mathbf n}$.
    We are interested in proving that the product $\chi_{\mathbf{m}}^\epsilon \cdot \chi_{\mathbf{n}}^\eta$ is an element of $A$. Depending on the tuples and signs, the product might be given by 
    $$
    \chi_{\mathbf{m}}^\epsilon \cdot \chi_{\mathbf{n}}^\eta = \chi_{(\mathbf{m}, \mathbf{n})}^{\epsilon},
    $$
    in which case one easily defines an admissible permutation of $(\mathbf m,\mathbf n)$ by simply considering the natural extension $\sigma$ of $\sigma_{\mathbf m}$ and $\sigma_{\mathbf n}$\footnote{Note that $f$ might differ by a factor of $-1$ on $\mathbf{n}$ when $\mathbf{n}$ is viewed as part of the pair $(\mathbf{m}, \mathbf{n})$, since the index matters. However, this does not affect admissibility.}. Another possibility, is that the product gives a sum of irreducible representations whose associated tuples are of the form 
    $$
    \mathbf m*_s^j \mathbf n = 
    (m_1,\cdots m_{p-j},s,n_{j+1},\cdots n_q),\quad 
    1\le j \le \min\{i : m_{p-i+1}\ne n_i \!\!\!\!\mod 2\},
    $$
    where $s$ is an integer of the same parity as $m_{p-j+1}+n_j$. Fix such a tuple $\mathbf m*_s^j\mathbf n$, call $\sigma$ the before mentioned admissible permutation of $(\mathbf m,\mathbf n)$. Define the sets $\mathrm {Odd}_0 := \mathrm {Odd}(\mathbf m,\mathbf n)$, $\mathrm {Odd}_j := \mathrm {Odd}(\mathbf m*_s^j\mathbf n)$ and 
    $$
    \mathrm {Odd}_i := \mathrm {Odd}(m_1,\cdots m_{p-i},0,n_{i+1},\cdots n_q) =  
    \mathrm {Odd}(\mathbf m *_0^i \mathbf n),\quad 1\le i \sl j.
    $$
    Inductively define admissible permutations by setting $\sigma_0 := \sigma$, and for each $1 \le i \sl j$:
\begin{itemize}
    \item If $m_{p-i+1}$ and $n_i$ are both even, then $\mathrm{Odd}_i = \mathrm{Odd}_{i-1}$ and let $\sigma_i := \sigma_{i-1}$, which is admissible for $\mathbf{m} *_0^i \mathbf{n}$.

    \item Otherwise, $m_{p-i+1}$ and $n_i$ are both odd, in which case we define:
    $$
    \sigma_i(\sigma_{i-1}(m_{p-i+1})) := \sigma_{i-1}(n_i) \quad \text{and} \quad \sigma_i(\sigma_{i-1}(n_i)) := \sigma_{i-1}(m_{p-i+1}),
    $$
    and extend $\sigma_i$ as $\sigma_{i-1}$ on the set 
    $$
    \mathrm{Odd}_i \setminus \{ \sigma_{i-1}(m_{p-i+1}), \sigma_{i-1}(n_i) \}.
    $$
\end{itemize}
The final rank may follow a similar case, in which the result is already obtained. However, if $m_{p-j+1}$ and $n_j$ are of different parities, then we define $\sigma_j$ to be equal to $\sigma_{j-1}$, except that the preimage of whichever is the odd number between $m_{p-j+1}$ and $n_j$ is now sent to $s$, yielding an admissible permutation of $\mathbf m *_s^j \mathbf n$.
\end{proof}

\section{The Limit Profile} \label{limprof}

Fix now $\psi^{(\alpha,\beta)} = (\psi_t)_{t\ge0}$ to be a Brownian motion of parameters $(\alpha,\beta)\in \R_+^*\times \R_+$ on $U_N^+$. 

\subsection{Absolute continuity}

To compute limit profiles, we only consider the total variation distance. Just like in the classical setting, the total variation distance for states is much easier to compute when one of the states is absolutely continuous w.r.t. the other. A state $\psi$ is said to be \textit{absolutely continuous} w.r.t. the Haar state if there exists $f\in \mathrm L^1(U_N^+)$ such that $\psi(x) = h(fx)$ for all $x\in \O(U_N^+)$. In this case, the total variation distance is easily expressed (see \cite[Lem 2.6]{Fre19})
$$
d_{\mathrm{TV}}(\psi,h) = \frac{1}{2}\Vert f-1\Vert_1.
$$
Quantum groups exhibit a unique characteristic where absolute continuity is not easily obtained. Let us first present an elementary fact before stating a proposition on absolute continuity. 

\begin{lem}\label{estimequad}
Let $\mathbf n \in \bigsqcup_{p\ge 1}\N^{*p}$ be any tuple. We have the following inequality: 
$$
0 \le \mathfrak e_{\mathbf n} \le \ell(\mathbf n)^2.
$$
Moreover, the upper bound is attained if and only if the tuple consists only of odd numbers. 
\end{lem}

\begin{proof}
We will use the notations introduced immediately before Lemma \ref{86461}. Let us set $x/y$ to be the amount of even/odd numbers within the set $\{k_i+i: 1\le i\le \mathfrak p_{\mathbf n}\}$. We thus have
\begin{align*}
    \mathfrak e_{\mathbf n} &= 
    \sum_{1\le i, j\le \mathfrak p_{\mathbf n}} (-1)^{k_j+j-(k_i+i)}
    \\ &= \left( 
    \sum_{i=1}^{\mathfrak p_{\mathbf n}} (-1)^{k_i+i}
    \right)^2\\ &=
    (x-y)^2,
\end{align*}
this shows the positivity. The upper bound is clear:
\begin{align*}
    \mathfrak e_{\mathbf n}  =(x-y)^2  &\le \max(x,y)^2 
    \le \mathfrak p_{\mathbf n}^2\le \ell(\mathbf n)^2.
\end{align*}
It is clearly attained when $\mathbf n$ consists only of odd numbers, as this would imply $k_i+i = 2i$ for all $1 \le i \le \mathfrak p_{\mathbf n}$, leading to $x = \mathfrak p_{\mathbf n} = \ell(\mathbf n)$ and $y=0$. If $\mathbf n$ has at least one even entry, simply observe that
$$
\max(x, y) \le \mathfrak{p}_{\mathbf{n}} \sl \ell(\mathbf{n}).
$$
\end{proof}

We can now characterize absolute continuity. 

\begin{prop}\label{abs cont}
Let $c\in \R$ and set $t_N = \alpha^{-1}(N\ln(\sqrt 2N)+cN)$, then 
\begin{itemize}
    \item if $c\sg 0$, then $\psi_{t_N}$ is absolutely continuous w.r.t. the Haar state, provided $N$ is large enough;
    \item if $c\sl 0$, then $\psi_{t_N}$ is not absolutely continuous w.r.t. the Haar state, provided $N$ is large enough. 
\end{itemize}
\end{prop}

\begin{proof}
Assume $c\sg 0$. We want to see that the series 
$$
f_{t_N} := 1+\sum_{\epsilon, \mathbf n}
d_{\mathbf n}
\exp\left(
-t_N\left(
\alpha
\lambda_{\mathbf n}+ \beta\frac{\mathfrak e_{\mathbf n}}{N}
\right)\right)\chi_{\mathbf n}^\epsilon
$$
converges in $\mathrm L^1(U_N^+)$ for $N$ large enough. We will make use of the $\mathrm L^2$-norm, as the $\mathrm L^1$-norm is dominated by it. Using the estimates from \cite[Lem 1.7]{FHLUZ17}, Lemma \ref{estimequad} and the fact that $d_n\le N^n$, we have
\begin{align*}
    \Vert f_t-1\Vert_2^2 &\le \sum_{\epsilon, \mathbf n}
d_{\mathbf n}^2e^{-2t_N\alpha\lambda_{\mathbf n}}
 \\ &\le 
2 \sum_{m\ge 1} \sum_{\mathbf n\in \Pi_m}
N^{2m}e^{-2t_N\alpha m/N}  \\
&\le 2\sum_{m\ge 1} \vert \Pi_m \vert e^{-2m(c+\ln \sqrt 2)} \\ &\le \sum_{m\ge 1} e^{-2mc}
\\
&\sl \infty.
\end{align*}
For the case $c\sl 0$, we restrict to the commutative subalgebra $ \mathcal{O}(U_N^{+})_{00}$ (see Proposition \ref{cond exp}), we may view the state $\psi_{t_N} $ and the Haar state as probability measures $\mu_{t_N}$ and $\nu_{\mathrm{SC}}$ respectively via the isomorphism $\mathcal{O}(H_N^{s+})_{00} \to \mathbb{C}[X], x \mapsto X$. Propositions \ref{cond exp} and \ref{measure description} then imply
\begin{align*}
    \mu_{t_N}(P_{2k}) &= \psi_{t_N}(x_{2k})
    \\ &= 
    \frac{2}{2^k} \sum_{\mathbf n\in \Pi_{2k}} \psi_{t_N}(\chi_{\mathbf n}^\epsilon ) \\
    &\ge \frac{2}{2^k} \sum_{\mathbf n\in A_{2k}^0} d_{\mathbf n} e^{-t_N\alpha \lambda_{\mathbf n}}
    \\ &\underset{N\to\infty}{\longto} \frac{1}{2^k}
    \binom{2k}{k} e^{-2k(c+\ln \sqrt 2)} \\
    &= \frac{e^{-2kc}}{2^{2k}} \binom{2k}{k} \\
    &= \eta_c^\infty(P_{2k}).
\end{align*}
This further implies that 
$$
\underset{N\to\infty}{\lim\inf}\,\, \mu_{t_N}(X^{2n}) \ge 
\eta_c^\infty(X^{2n}),\quad n\in \N.
$$
This is because $X^{2n}$ can be expressed as a linear combination of $P_{2k}$'s with positive coefficients. Now, since $\eta_c^\infty$ allocates mass outside of $[-2,2]$ (see Proposition \ref{measure description}), then so must $\mu_{t_N}$ for $N$ large enough. Recall that through the considered isomorphism, the Haar state corresponds to the semicircle distribution whose support is $[-2,2]$, hence absolute continuity does not hold. 
\end{proof}

\subsection{Computation of the limit profile}

As stated in the previous section, we would like to restrict the study of the Brownian motion to the subalgebra $\O(U_N^+)_{00}$. Unfortunately, the Brownian motion is not $\F$-invariant, note for instance that 
$$
\psi_t(\chi_2) = 
(N^2-1)e^{-2tN\alpha/(N^2-1)} \ne  N^2e^{-2t(N\alpha+2\beta)/N^2} = \psi_t (\chi_{(1,1)}).
$$
We will make use of the process defined by $\widetilde \psi_t = \psi_t\circ \F$. 

\begin{thm}\label{thisisit}
    Let $\psi^{(\alpha,\beta)} = (\psi_t)_{t\ge0}$ be a Brownian motion of parameters $(\alpha,\beta)\in \R_+^2$ on $U_N^+$, assume that $\alpha\sg 0$ and that 
    $$
    \frac{\beta\ln N}{\alpha} \underset{N\to\infty}{\longto} r,
    $$
    for some $r\in [0,\infty]$. Then, $\psi^{(\alpha,\beta)}$ has cutoff at time $\alpha^{-1}N\ln N$. More precisely, it exhibits the following cutoff: 
    $$
d_{\mathrm{TV}}
\big(\widetilde \psi_{\alpha^{-1}(\ln (\sqrt 2N)+cN},h\big) \underset{N\to\infty}{\longto} d_{\mathrm TV} (\eta_c^r,\nu_{\mathrm{SC}}),\quad c\sg 0
$$
and 
$$
\underset{N\to\infty}{\lim\inf}\, d_{\mathrm{TV}}\big(\widetilde \psi_{\alpha^{-1}(\ln (\sqrt 2N)+cN},h\big) \ge  d_{\mathrm TV} (\eta_c^r,\nu_{\mathrm{SC}}),\quad c\in \R,
$$
where $\nu_{\mathrm{SC}}$ is the semicircle distribution and $\eta_c^r$ is the probability measure introduced in Lemma \ref{what are the measures}.
\end{thm}

\begin{proof}
    Up to renormalising, we may assume that $\alpha = 1$. We first prove the result for the process $(\widetilde \psi_t)_{t\ge0} = (\psi_t\circ\F)_{t\ge0}$.
    The process being $\F$-invariant, we may restrict the study to the commutative algebra $\O(U_N^+)_{00}$ for which the $x_m$'s form an orthonormal basis. 

    Let $c\sg 0$ and set $t_N = N\ln (\sqrt 2N) + cN$. It follows from Proposition \ref{abs cont} that $\widetilde \psi_{t_N}$ has an $\mathrm L^1$-density (at least asymptotically defined) given by 
    $$
    \widetilde f_{t_N} = 
    \sum_{m\ge 0} {\widetilde \psi_{t_N}(x_m)}
    x_m
    =1+\sum_{m\ge 1} \frac{1}{\sqrt 2^{m}} 
    \left( \sum_{\substack{\mathbf n\in \Pi_m,\\ \epsilon = \pm1}}
    \psi_{t_N}(\chi^\epsilon_{\mathbf n})
    \right)x_m.
    $$
    Using the isomorphism $\O(U_N^+)_{00}\to \C[X],x_m\mapsto P_m$ sending the Haar state to the semicircle distribution, we have 
    $$
    d_{\mathrm {TV}}(\widetilde \psi_{t_N},h) =
    \frac{1}{2} 
    \Vert \widetilde f_{t_N}-1\Vert_{1,h}
    = \frac{1}{2}\left\Vert
    \sum_{m\ge 1} \widetilde \psi_{t_N}(x_m)P_m
    \right\Vert_{1,\nu_{\mathrm{SC}}}
    .
    $$
    In addition, Proposition \ref{measure description} implies that:
    $$
    \widetilde \psi_{t_N}(x_m) \underset{N\to \infty}{\longto} 
    e^{-cm}\E[e^{-rS_n^2}],\quad m\in \N,
    $$
    where $(S_n)_{n\ge0}$ is a $\pm1$ walk. 
    Moreover, for any fixed $m\in \N$, we have using yet again the estimates from \cite[Lem 1.7]{FHLUZ17}, Lemma \ref{estimequad} and the fact that $d_n\le N^n$,
\begin{align*}
    \Vert 
    \widetilde \psi_{t_N}(x_m)x_m
    \Vert_1 &\le \Vert 
    \widetilde \psi_{t_N}(x_m)x_m
    \Vert_2 \\ &=
    \frac{1}{\sqrt 2^{m}}  \sum_{\substack{\mathbf n\in \Pi_m,\\ \epsilon = \pm1}}
    \psi_{s_t}(\chi^\epsilon_{\mathbf n}) \\ 
    &\le \frac{1}{\sqrt 2^{m}}  \sum_{\substack{\mathbf n\in \Pi_m,\\ \epsilon = \pm1}} d_{\mathbf n}e^{-t_N\lambda_{\mathbf n}} \\ &\le 
    \sqrt 2^{m} N^m e^{-m(\ln (\sqrt 2N)+c)} \\
    &= e^{-cm}.
\end{align*}
The bound being uniform in $N$ allows for an exchange of the sum over $m$ and of the limit in $N$, yielding using Lemma \ref{what are the measures} and Proposition \ref{measure description}:
\begin{align*}
    \lim_{N\to\infty} d_{\mathrm {TV}}(\widetilde \psi_{t_N},h) &=
    \frac{1}{2} \left\Vert
    \sum_{m\ge1}e^{-cm}\E[e^{-rS_n^2}]P_m
    \right\Vert_{1,\nu_{\mathrm {sc}}} \\
    &= \frac{1}{2} \left\Vert
    \sum_{m\ge1}\eta_c^r(P_m)P_m
    \right\Vert_{1,\nu_{\mathrm {sc}}} \\ &= 
    d_{\mathrm{TV}}(\eta_c^r,\nu_{\mathrm{SC}})
\end{align*}

Now, for the second result, fix $c \in \R$, set $t_N = N\ln(\sqrt 2N)+cN$ and consider the states $\widetilde{\psi}_{t_N}$ and $h$ in their measure form as $\mu_{t_N}$ and $\nu_{\mathrm{SC}}$. As already mentioned, we have 
\begin{align*}
    \mu_{t_N}(P_m) \le e^{-cm} = \eta_c^0(P_m),\quad m\in \N.
\end{align*}
In particular, this implies that:
$$
\mu_{t_N}(X^{2n}) \leq \eta^0_{c}(X^{2n}), \quad n \in \mathbb{N}.
$$
This follows from the fact that $X^{2n}$ can be expressed as a linear combination of $P_m$'s with positive coefficients. Since $\eta_{c}^0$ is supported on the interval $[-\gamma_{c}, \gamma_{c}]$, where $\gamma_{c} = e^{c} + e^{-c}$, it follows that $\mu_{t_N}$ must also be supported on this interval (as the support of a measure can be estimated by bounding its moments. See for example the beginning of \cite[Lect 2]{Sch20}). Given that moment convergence with compact support implies weak convergence, we have:
$$
\mu_{t_N}(B) \underset{N\to\infty}{\longto} \eta_c^r(B),
$$
for any continuity set $B$ of $\eta_c$, i.e., a set such that $\eta_c(\partial B) = 0$. Since $f_c^r$ is continuous, the set $B = \{f_c^r \sl 1\}$ is open and maximizes the total variation distance, leading to
\begin{align*}
    d_{\mathrm{TV}}(\widetilde \psi_{t_N},h) &\ge d_{\mathrm{TV}}(\mu_{t_N},\nu_{\mathrm{SC}}) \\
    &\ge \nu_{\mathrm{SC}}(B) - \mu_{t_N}(B) \\
    &\underset{N\to\infty}{\longto} \nu_{\mathrm{SC}}(B) - \eta_c^r(B) \\ &= d_{\mathrm{TV}}(\nu_{\mathrm{SC}}, \eta_c^r). 
\end{align*}
Now we make the link with the actual Brownian motion $\psi^{(\alpha,\beta)}$. 

For the first convergence (setting $t_N = N\ln(\sqrt 2N)+cN$ for fixed $c\sg 0$) it suffices to see that $
d_{\mathrm{TV}}\big(\psi_{t_N},\widetilde \psi_{t_N}\big)\to 0$ as $N\to \infty$. This is obvious as 
$$
d_{\mathrm{TV}}\big(\psi_{t_N},\widetilde \psi_{t_N}\big) \le 
\sum_{\epsilon,\mathbf n} \vert 
\psi_{t_N}(\chi_{\mathbf n}^{\epsilon})
-\widetilde \psi_{t_N}(\chi_{\mathbf n}^{\epsilon})
\vert^2 \le \Vert 
 f_{t_N}-1
\Vert_2^2 + \Vert 
\widetilde f_{t_N}-1
\Vert_2^2 \le 2\sum_{m\ge 1}e^{-2cm}
$$
where the last inequality holds for $N$ large enough (similarly to the proof of Proposition \ref{abs cont}). By dominated convergence, we may exchange summation and limit in $N$ yielding the convergence as $\psi^{(\alpha,\beta)}$ and $\widetilde \psi$ have same moment convergence. 

For the second, simply note that
$$
d_{\mathrm {TV}}(\psi_{t},h) = \frac{1}{2} \Vert \psi_{t_c}-h\Vert_{FS} \ge 
\frac{1}{2} \Vert \F\circ \psi_{t}-\F\circ h\Vert_{FS} = d_{\mathrm {TV}}(\widetilde \psi_{t},h).
$$
\end{proof}

Let us make a few comments on the parameter $r = \lim_N \beta\ln N/\alpha$. First, note that even if no such convergence occurs, then one still has cutoff at the same time and now different subsequences may have different cutoff profiles. However, we emphasize that it is quite natural to consider Brownian motions where such a parameter exists. 

Second, recall that (see Remark \ref{unique profile}) cutoff profiles are only unique up to affine transformation. The end of this subsection is dedicated to proving that every $r\in [0,\infty]$ does define a unique cutoff profile. 

We first estimate the weight of the singular part of the measure $\eta_c^r$. 

\begin{lem}\label{integral}
    For every $r\in (0,\infty]$ set $R_r = \cos T_r$ where $T_r$ is a random variable of Gaussian distribution $\mathcal N(0,2r)$ if $r\in (0,\infty)$, and $T_\infty \sim \mathrm{Unif}([0,2\pi])$. Then, for every $r\in (0,\infty)$, we may write
    $$
    \E\left[
    \left(1-\epsilon^2R_r^{-2} \right)_+
    \right] = 1+A_r\epsilon + B_r\epsilon^2 + o(\epsilon^2),\quad \epsilon\downarrow0,
    $$
    where 
    $$
    A_r = -\frac{4}{\pi} \frac{1-e^{-4r^2}}{1+e^{-4r^2}}\emph{\and}  B_r = \frac{4e^{-4r^2}}{(1+e^{-4r^2})^2}.
    $$
    Moreover, 
    $$
    \E\left[
    \left(1-\epsilon^2R_r^{-2} \right)_+
    \right] = \frac{2\left( 
    \arccos \epsilon - \epsilon\sqrt{1-\epsilon^2}
    \right)}{\pi} =
    1-\frac{4\epsilon}{\pi} + \frac{\epsilon^3}{3\pi} + O(\epsilon^5).
    $$
\end{lem}

We prove that every parameter $r\in [0,\infty]$ defines a unique cutoff profile by using the weight of the singular part. 

\begin{proof}
Fix $r\in (0,\infty)$ and set\footnote{This is the Poisson summation formula for the heat kernel, see \cite{Woi20}.}
    $$
    \phi_r(\theta) = 
    \frac{1}{\sqrt{4r\pi}} \sum_{k\in \Z} e^{-(\theta+k\pi)^2/4r} = \frac{1}{\pi} \left( 
    1+2\sum_{n\ge1} e^{-4r^2n}\cos (2n\theta)
    \right).
    $$ 
Note the following properties of the function $\phi_r$ which we will repetitively use through the rest of the proof:
    \begin{itemize}
        \item $\phi_r$ is smooth;
        \item $\phi_r$ is even and $\pi$-periodic;
        \item $\phi_r(\theta)\mathbf 1_{\vert\theta\vert\sl \pi/2}\d\theta$ defines a probability measure. 
    \end{itemize}
    Now compute with $\epsilon\sg0$:
    \begin{align*}
        \E\left[\left(1-\epsilon^{2}R_r^{-2}\right)_+\right] &=\frac{1}{\sqrt{4\pi r}} \int_{\{\vert\cos t\vert \sg \epsilon\}} \left(
        1-\frac{\epsilon^2}{\cos (t)^2}
        \right) e^{-t^2/4r}\d t \\
        &= \int_{-\pi/2}^{\pi/2}
        \left( 
        1-\frac{\epsilon^2}{\cos (\theta)^2}
        \right) \phi_r(\theta)
        \mathbf 1_{\{\vert \cos t\vert \sg \epsilon\}}
        \d \theta \\
        &= \int_{-\arccos \epsilon}^{\arccos \epsilon}
        \left( 
        1-\frac{\epsilon^2}{\cos (\theta)^2}
        \right) \phi_r(\theta)
        \d \theta \\
        &= 
        I_r^\epsilon - \epsilon^2(J_r^\epsilon + K_r^\epsilon),
    \end{align*}
    with
    \begin{align*}
        I_r^\epsilon &= \int_{-\arccos \epsilon}^{\arccos \epsilon}
        \phi_r(\theta)
        \d \theta \\ &= 1-2 \int_{\arccos \epsilon}^{\pi/2}\phi_r(\theta)\d \theta \\
        &= 1 - 2 \int_0^{\arcsin \epsilon} \phi_r\left( 
        \frac{\pi}{2}-\theta
        \right)\d \theta \\ 
        &= 1 - 2 \int_0^{\arcsin \epsilon} \phi_r\left( 
        \frac{\pi}{2}+\theta
        \right)\d \theta \\ &= 
        1-2\left( 
        \phi_r\left(\frac{\pi}{2}\right)\arcsin \epsilon + 
        \frac{\phi_r'(\pi/2)}{2}\arcsin (\epsilon)^2 +O\left(\arcsin (\epsilon)^3\right) 
        \right) \\ 
        &= 1-2\phi_r\left( \frac{\pi}{2}\right) \epsilon+O(\epsilon^3).
    \end{align*}
    From the third to fourth line we have used the fact that $\phi_r'(\pi/2) = 0$. Now for the two last terms
    \begin{align*}
        J_r^\epsilon &= \int_{-\arccos \epsilon}^{\arccos \epsilon} \frac{\phi_r(\pi/2)}{\cos(\theta)^2}\d\theta \\ &=
        2\phi_r\left( 
        \frac{\pi}{2}
        \right) \tan\big( 
        \arccos\epsilon
        \big)\\ &= 
        2\phi_r\left( 
        \frac{\pi}{2}
        \right) \frac{\sqrt{1-\epsilon^2}}{\epsilon} \\
        &= \frac{2}{\epsilon}\phi_r\left( 
        \frac{\pi}{2}
        \right) + O(\epsilon),
    \end{align*}
    and 
    \begin{align*}
        K_r^\epsilon &= \int_{-\arccos \epsilon}^{\arccos \epsilon} \frac{\phi_r(\theta)-\phi_r(\pi/2)}{\cos(\theta)^2}\d\theta \\
        &= 2\int_{0}^{\arccos \epsilon} \frac{\phi_r(\theta)-\phi_r(\pi/2)}{\cos(\theta)^2}\d\theta \\
        &= 2\int_{0}^{\pi/2} \frac{\phi_r(\theta)-\phi_r(\pi/2)}{\cos(\theta)^2}\d\theta - 2\int_{\arccos \epsilon}^{\pi/2} \frac{\phi_r(\theta)-\phi_r(\pi/2)}{\cos(\theta)^2}\d\theta \\
        &= \frac{4}{\pi} \sum_{n\ge1}e^{-4r^2n}\int_0^{\pi/2}
        \frac{\cos(2n\theta)-(-1)^n}{\cos(\theta)^2}\d\theta 
        + O(\epsilon) \\ &= 
        \frac{4}{\pi} \sum_{n\ge1}e^{-4r^2n}\int_0^{\pi/2}
        \frac{\cos(2n\theta)-(-1)^n}{\cos(\theta)^2}\d\theta 
        + O(\epsilon) \\
        &= 4\sum_{n\ge1} ne^{-4r^2n}(-1)^n+O(\epsilon).
    \end{align*}
    From the fifth to sixth line we have computed the following integral: 
    \begin{align*}
        \int_0^{\pi/2} \frac{\cos(2n\theta)-(-1)^n}{\cos(\theta)^2} \d\theta 
        &= \int_0^{\pi/2} \frac{\cos(2n(\pi/2-\theta))-(-1)^n}{\cos(\pi/2-\theta)^2} \d\theta \\ &=
        \int_0^{\pi/2} \frac{(-1)^n\cos(2n\theta)-(-1)^n}{\sin(\theta)^2} \d\theta \\
        &= (-1)^n \int_0^{\pi/2} \frac{\cos(2n\theta)-1}{\sin(\theta)^2} \d\theta 
        \\ &= 
        (-1)^{n+1} \int_0^{\pi/2} \frac{2\sin(n\theta)^2}{\sin(\theta)^2}\d\theta \\
        &= (-1)^{n+1} \int_{-\pi/2}^{\pi/2}
        \left(\frac{e^{\mathrm in\theta}-e^{-\mathrm in\theta}}{e^{\mathrm i\theta}-e^{-\mathrm i\theta}}
        \right)^2\d\theta \\ &=
        (-1)^{n+1} \int_{-\pi/2}^{\pi/2}
        e^{2\mathrm i(n-1)\theta} \left( 
        \frac{1-e^{-2\mathrm in\theta}}{1-e^{-2\mathrm i\theta}}
        \right)^2\d\theta \\ 
        &= (-1)^{n+1} \int_{-\pi/2}^{\pi/2}
        e^{2\mathrm i(n-1)\theta} \sum_{k,\ell = 0}^{n-1}e^{-2\mathrm i(k+\ell)\theta}\d\theta \\ &=
        (-1)^{n+1}\pi n. 
        \end{align*}
        Putting everything together gives the result for $r$ finite:
        \begin{align*}
            A_r &= -4\phi_r\left( 
            \frac{\pi}{2}
            \right) \\&= 
            -\frac{4}{\pi} \left( 
            1+2\sum_{n\ge 1 }(-1)^n e^{-4r^2n}
            \right)
            \\ &= -\frac{4}{\pi} \left( 
            1-2\frac{e^{-4r^2}}{1+e^{-4r^2}}
            \right)\\
            &= -\frac{4}{\pi} \frac{1-e^{-4r^2}}{1+e^{-4r^2}},
        \end{align*}
        and 
        \begin{align*}
            B_r &= -4\sum_{n\ge 1 }ne^{-4r^2n}(-1)^n \\
            &= \frac{4e^{-4r^2}}{(1+e^{-4r^2})^2}.
        \end{align*}
        We now treat the infinite case. 
        \begin{align*}
            \E\left[\left(1-\epsilon^{2}R_\infty^{-2}\right)_+\right] &= \frac{1}{\pi}
            \int_{-\pi/2}^{\pi/2} \left( 
            1-\frac{\epsilon^2}{\cos(\theta)^2}
            \right)_+ \d\theta \\ &= 
            \frac{1}{\pi} \int_{-\arccos \epsilon}^{\arccos\epsilon}
            \left( 
            1-\frac{\epsilon^2}{\cos(\theta)^2}
            \right) \d\theta \\ 
            &= 
            \frac{2}{\pi} \left( 
            \arccos \epsilon - \epsilon^2\tan( 
            \arccos \epsilon 
            )
            \right)
            \\ &=
            \frac{2}{\pi} \left( 
            \arccos \epsilon -
            \epsilon\sqrt{1-\epsilon^2}\right)
            \\
            &=
            1-\frac{4\epsilon}{\pi} + \frac{\epsilon^3}{3\pi} + O(\epsilon^5).
        \end{align*}
\end{proof}

\begin{prop}\label{unique profile}
    Every $r\in [0,\infty]$ defines a unique cutoff profile through the mapping 
    $$
    g_r:
    c\longmapsto d_{\mathrm{TV}}(\nu_{\mathrm{SC}},\eta_c^r).
    $$
\end{prop}

\begin{proof}
     Fix $r\in [0,\infty]$. First recall (see Proposition \ref{what are the measures}) that we may write 
     $$
     \eta_c^r = f_c^r \d\nu_{\mathrm{SC}} + \eta_{c,\mathrm{sing}}
     $$
     with $\eta_{c,\mathrm{sing}}^r \perp \nu_{\mathrm{SC}}$ and 
    $$
    f_c^r(x) =
    \E\left[\frac{1}{e^{-2c}R_r^2-xe^{-c}R_r+1}\right]
    \and \eta_{c,\mathrm{sing}}^r(A) = \E\left[\left(
    1-e^{2c}R_r^{-2}
    \right)_+ \mathbf 1_A(e^{-c}R_r + e^cR_r^{-1})\right],
    $$
    where $R_r = \cos(T_r)$ for some random variable of Gaussian distribution $T_r\sim \mathcal N(0,2r)$ (with $T_0 = 0$ and $T_\infty\sim \mathrm{Unif}([0,2\pi])$). Since $f_c^r\downarrow 0$ uniformly as $c\to -\infty$, the total variation distance between $\eta_c^r$ and $\nu_{\mathrm{SC}}$ is carried out by the weight of the singular part for $c$ close enough to $-\infty$. In other words, $g_r$ coincides with $c\mapsto \eta_{c,\mathrm{sing}}^r(\R)$ in an interval of the form $(-\infty,c_r)$. 

    Fix $r,s\sg0$ and assume that $g_r$ and $g_s$ are equivalent modulo affine transformation, i.e. 
    $$
    g_r(ac+b) = g_s(c),\quad c\in \R,
    $$
    for some $a\sg 0$ and $b\in \R$. Following Lemma \ref{integral}, this implies 
    $$
    1+A_re^{ac+b} + B_re^{2(ac+b)}+O(e^{3(ac+b)})
    = 
    1+A_se^{c} + B_se^{2c}+ O(e^{3c}),\quad c\to\infty. 
    $$
    This already implies $a = 1$, we are left with the system 
    $$
    \left\{ 
    \begin{array}{cc}
        A_re^b &= A_s  \\
        B_re^{2b} &= B_s 
    \end{array}
    \right.
    $$
    This implies $A_r^2/B_r = A_s^2/B_s$ which forces $r = s$ as the function 
    $$
    x\longmapsto \frac{A_x^2}{B_x} = \frac{4}{\pi^2}(e^{2x^2}-e^{-2x^2})^2,
    $$
    is injective on $\R_+$. 

    Assume $g_0$ and $g_r$ (for $r\in (0,\infty)$) are equivalent modulo affine transformation, this would imply (using again Lemma \ref{integral})  
    $$
    1-e^{2(ac+b)} = 1+A_se^{c} + B_se^{2c}+ O(e^{3c})
    ,\quad c\to-\infty.
    $$
    for some $a\sg 0$ and $b\in \R$, which is impossible. So $g_0$ is not equivalent modulo affine transformation to $g_r$ for $r\ne 0$, similarly it is not equivalent modulo affine transformation to $g_\infty$. 
    
    Assume now that $g_r$ and $g_\infty$ are equivalent modulo affine transformation for some $r\in (0,\infty)$. This implies, using Lemma \ref{integral}:
    $$
    1-\frac{4e^c}{\pi} + \frac{e^{3c}}{3\pi}+ O(e^{5c})=
    1+A_re^{ac+b} + B_re^{2(ac+b)}+ O(e^{3(ac+b)})
    ,\quad c\to-\infty.
    $$
    for some $a\sg 0$ and $b\in \R$. The second order implies $a = 1$ and the third $a = 3/2$ which is a contradiction. 
\end{proof}

Below is a plot of the different cutoff profiles $g_r$ ($0\le r\le \infty$). 

\begin{figure}[htbp]
    \centering
    \includegraphics[width=0.75\linewidth]{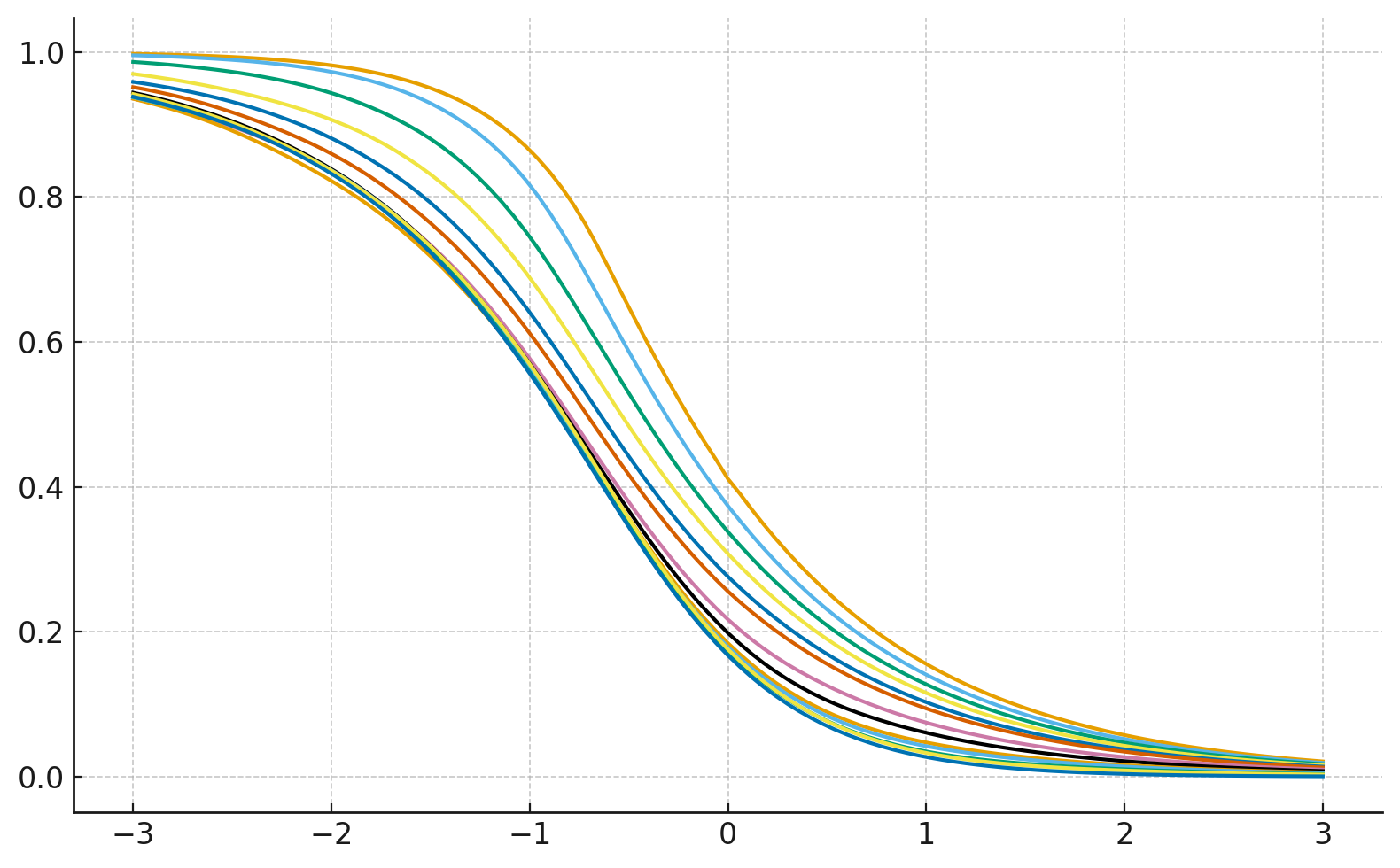}
    \caption{Total variation distance $g_r$ for increasing $r$ (top to bottom).}
    \label{fig:gr_semicircle}
\end{figure}

Let us conclude this section by discussing the result of Theorem \ref{thisisit}. The limit profile on the right ($c > 0$) is straightforward to compute thanks to the absolute continuity in this region. However, outside of this region, absolute continuity is lost. It is worth recalling that such a loss of absolute continuity was already observed in \cite{FTW21}, where all processes studied exhibited an atom as the singular part when absolute continuity was lost. P. Biane, in \cite[Sec 12.2]{Bia08}, introduced an interesting method for recovering the absolutely continuous part of a measure when only its moments are known, provided that one can identify its atom. In our case, the singular part is more complex than a single atom, and it remains unclear whether Biane’s method could be applied here. We conjecture, however, that the singular part lies outside the support of $\nu_{\mathrm{SC}}$, and that there is true $\mathrm{L}^1$-convergence. Specifically, we believe that even for $c < 0$, the absolutely continuous part of the process converges to that of $\eta_c$ in $\mathrm{L}^1(\nu_{\mathrm{SC}})$. Nonetheless, we have reasons to believe that the density of the absolutely continuous part (when $c\sl 0$) is in $\mathrm{L}^1(\nu_{\mathrm{SC}})$ (by definition) but not in $\mathrm{L}^2(\nu_{\mathrm{SC}})$, which renders Biane's method inapplicable here.

\section*{Appendix A} \label{Appendix A}

In this appendix, we give a description of Gaussian centralized generating functionals on $O_N^+$ and $U_N^+$. We will therefore focus our study on the latter quantum groups, we refer the reader to \cite[Sec 1.5]{FS16} for definitions in greater generality. 

We would like to acknowledge that this result was previously achieved by different means by Franz, Freslon, and Skalski in \cite{FFS24} (unavailable at the time of the first version of this work's preprint) in which the authors provide a comprehensive description of Gaussian generating functionals on $O_N^+$ and $U_N^+$.

\begin{deB}
    Let $\G = O_N^+$ or $U_N^+$ and $D$ a pre-Hilbert space. We call \textit{Schürmann triple} on $\G$ over $D$ a family of linear maps $(\rho,\eta,L)$ where 
    \begin{itemize}
        \item $\rho:\O(\G)\to \mathcal L(D)$\footnote{$
        \mathcal L(D) = \{
        X: D\to D \hbox{ linear } \vert \,\exists 
        X^*:D\to D \hbox{ s.t. }
        \langle u,Xv\rangle = \langle X^*u,v\rangle, \,\, \forall u,v\in D
        \}
        $.} is a unital $*$-homomorphism;
        \item $\eta:\O(\G)\to D$ satisfies 
        $$
        \eta(xy) = \rho(x)\eta(y) + \eta(x)\varepsilon(y),\quad x,y\in \O(\G);
        $$
        \item $L:\O(\G)\to \C$ is hermitian and satisfies
        $$
        L(xy) = \varepsilon(x)L(y) + L(x)\varepsilon(y) + \langle \eta(x^*),\eta(y)\rangle,\quad x,y\in \O(\G).
        $$
    \end{itemize}
\end{deB}
The definition implies that $L$ is a generating functional. Two Schürmann triples $(\rho_i,\eta_i,L)$ ($i=1,2$) on $\G$ and over pre-Hilbert spaces $D_i$ are said to be \textit{equivalent}, if there exists a surjective isometry $U:D_1\to D_2$ s.t. 
    $$
    \eta_2(x) = U\eta_1(x) \and \rho_2(x) U = U\rho_1(x),\quad x\in \O(\G).
    $$
There is a one-to-one correspondence between Schürmann triples with surjective cocycle (up to equivalence) and generating functionals. 

\begin{propB}
    Let $L$ be a generating functional on $\G = O_N^+$ or $U_N^+$. Then the following are equivalent 
    \begin{enumerate}[label = \emph{(\roman*)}]
        \item $L(xyz) = L(xy)\varepsilon(z)+L(xz)\varepsilon(y)+L(yz)\varepsilon(x)-\varepsilon(xy)L(z)-\varepsilon(xz)L(y)-\varepsilon(yz)L(x)$ for all $x,y,z\in \O(\G)$;
        \item $\eta(xy) =\varepsilon(x)\eta(y) + \eta(x)\varepsilon(y)$ for all $x,y\in \O(\G)$.
    \end{enumerate}
\end{propB}

If a Schürmann triple satisfies one of these conditions, then we call it \textit{Gaussian}. The Gaussian property of a generating functional gives a way to compute its values on any product of elements. 

\begin{lemB}\cite[Prop 2.7]{FFS23}\label{gaussian}
    Let $L$ be a Gaussian generating functional on $\G = O_N^+$ or $U_N^+$. Then, for any $x_1,\cdots ,x_n\in \O(\G)$ we have 
    $$
    L(x_1\cdots x_n) = \sum_{1\le i\sl j\le n}L(x_ix_j)\varepsilon(
    x_0\cdots \Check x_i\cdots \Check x_j\cdots x_n
    ) - (n-2) \sum_{j=1}^n L(x_j)\varepsilon(
    x_1\cdots \Check x_j\cdots x_n
    ).
    $$
\end{lemB}

\begin{propB}\label{Gaussian decompo}
    We have the following decompositions. 
    \begin{itemize}
        \item Any centralized Gaussian generating functional on $O_N^+$ is of the form 
        $$
        \chi_n\longmapsto -bP_n'(N) 
        $$
        for some $b\ge 0$. 
        \item Any centralized Gaussian generating functional on $U_N^+$ is of the form
    $$
    \chi_w \longmapsto -
    \left(
    \ell(w)\alpha + \big( 
    p(w)-q(w)
    \big)^2\beta + \big(
    p(w)-q(w)
    \big)\zeta
    \right)
    N^{\ell(w)-1},\quad w\hbox{ word on } \{\lozenge,\blacklozenge\},
    $$
    for some $\alpha, \beta\ge 0$, and $\zeta\in \mathrm i\R$.
    \end{itemize}
\end{propB}

\begin{proof}
    Let us first note that any Gaussian cocycle $\eta$ defined on a compact quantum group satisfies the relation $\eta\circ S + \eta = 0$ (see the beginning of the proof of \cite[Thm 3.11]{FFS23}). 

    Let $(\rho,\eta,L)$ be a Gaussian Schürmann triple on $O_N^+$. The relation $\eta\circ S = - \eta$ implies that $\eta\equiv 0$ on the central algebra and therefore: 
    $$
    L(xy) = \varepsilon(x) L(y) + L(x)\varepsilon(y),\quad x,y\in \O(O_N^+)_0.
    $$
    It follows from a quick induction that $L$ is of the required form for $b := -L(\chi_1)$.

    Let $(\rho,\eta,L)$ be a Gaussian Schürmann triple on $U_N^+$. The relation $\eta\circ S = -\eta$ implies that $\eta +\overline \eta \equiv 0$ on the central algebra and therefore, we have 
    $$
    L(xy) = \varepsilon(x)L(y)+L(x)\varepsilon(y) - \langle \eta(x),\eta(y)\rangle,\quad x,y\in \O(U_N^+)_0.
    $$
    It follows from a quick computation that $L$ is of the required form for words of length $\le 2$ with 
    $$
    \alpha = -\Re L(\chi_\lozenge) - \frac{\Vert \eta(\chi_{\lozenge})\Vert^2}{2N}
    ,\quad \beta = \frac{\Vert \eta(\chi_{\lozenge})\Vert^2}{2N}
    \and \zeta = -\mathrm i \Im L(\chi_\lozenge).
    $$
    Let $w = w_1\cdots w_n$ be a word of length $n\ge 3$. To ease notations, we set $\ell := \ell(w)$, $p := p(w)$ and $q := q(w)$. By Lemma \ref{gaussian}, we compute 
    \begin{align*}
        L(w) &= \sum_{1\le i\sl j\le n} L(\chi_{w_iw_j})N^{n-2} - (n-2) 
        \sum_{i=1}^n L(\chi_{w_i})N^{n-1} \\ &= 
        \left(pq L(\chi_{\lozenge\blacklozenge}) + \frac{p(p-1)}{2}L(\chi_{\lozenge\lozenge}) + \frac{q(q-1)}{2}L(\chi_{\blacklozenge\blacklozenge})\right)N^{n-2} \\
        &\,\,\,\,\,\, - (n-2)\left( 
        pL(\chi_\lozenge) + qL(\chi_\blacklozenge)
        \right)N^{n-1} \\ &=-
        \Big(
        n(n-1) - n(n-2)
        \Big)\alpha N^{n-1} \\
        &\,\,\,\,\,\, - \Big(
        2p(p-1)+2q(q-1) - n(n-2)
        \Big)\beta N^{n-1} \\ 
        &\,\,\,\,\,\, -  \Big( 
        p(p-1)+q(q-1) - (n-2)(p-q)
        \Big)\zeta N^{n-1} \\ &= -\Big( 
        n\alpha + \big(p-q\big)^2\beta + \big(p-q\big)\zeta
        \Big) N^{n-1}.
    \end{align*}
    Which proves the result. 
\end{proof}

We conclude by explaining why we restrict our study to real-valued (centralized) Gaussian generating functionals on $U_N^+$. If we were to consider the imaginary part, convergence may not be guaranteed unless a bound exists that links the real and imaginary parts, similarly to the relation between $\alpha$ and $\beta$. However, we emphasize that, by introducing a \textit{drift} to the generating functional, the imaginary part can always be disregarded.

\begin{remB}
    A \textit{drift} on $U_N^+$ is defined as a generating functional whose associated cocycle is zero. In other words, a drift is an $\varepsilon$-derivation, i.e., a functional $ D: \mathcal{O}(U_N^+) \to \mathbb{C} $ satisfying
    $$
    D(xy) = \varepsilon(x)D(y) + D(x)\varepsilon(y), \quad x,y \in \mathcal{O}(U_N^+)_0.
    $$
\end{remB}

Note that a drift is completely determined by its values on the generators. More precisely, there is a Lie algebra isomorphism between the algebra of drifts and the Lie algebra $ \mathfrak{u}_N $ of the classical group $ U_N $. For each skew-Hermitian matrix $ H \in \mathfrak{u}_N $, we associate a drift $ D_H $ by setting $ D_H(u_{ij}) = H_{ij} $. The Lie bracket is given by
$$
[D_H, D_K] = D_H \star D_K - D_K \star D_H = D_{[H,K]}.
$$
This shows that drifts correspond to the classical part of Gaussian generating functionals. By adding such a drift\footnote{a centralized drift $D_H$ corresponds to a centralized generating functional as described in Proposition \ref{Gaussian decompo} with $\alpha= \beta = 0$ and $\zeta = \Tr H$.}, we can always assume a centralized Gaussian generating functional to be real-valued.

\begin{remB}
    Note that every pair $(\alpha,\beta) \in \R_+^2$ defines a generating functional $L_{\alpha\beta}$ in the sense of Formula \eqref{word}. Indeed, the functionals with $\beta = 0$ are obtained by composing a Brownian motion on $O_N^+$ (see Equation \eqref{centralON+} with $\nu=0$) with the quotient map $\O(U_N^+)\to \O(O_N^+)$ and the conditional expectation $\E:\O(U_N^+)\to \O(U_N^+)_0$, while the ones with $\alpha = 0$ are obtained by following the classification of centralized Gaussian generating functionals on $U_N^+$ in \cite{FFS24} and considering the case where the construction is made with scalar matrices. All such functionals exist, as the set of all generating functionals forms a convex cone.
\end{remB}



\bibliographystyle{plain}
\bibliography{mybio}
\end{document}